\documentclass[english]{article}

\usepackage{lmodern}
\usepackage[a4paper]{geometry}

\usepackage[utf8]{inputenc}
\usepackage[T1]{fontenc}
\usepackage{amsthm}
\usepackage{amsmath} 
\usepackage{amsfonts}
\usepackage[linesnumbered,ruled,vlined]{algorithm2e}
\usepackage{hyperref}
\usepackage{graphicx}
\usepackage[dvipsnames]{xcolor}
\usepackage{multirow}
\usepackage[square,numbers]{natbib}
\SetKwInput{KwInput}{Input}               
\SetKwInput{KwOutput}{Output}    
\SetKwFunction{KwFn}{Function}           
\usepackage{authblk}
\usepackage{tikz,pgf}
\usepackage{tkz-euclide}
\usepackage{tkz-graph}
\usepackage{bm}
\usetkzobj{all} 
\usepackage{eqnarray}
\newtheorem{theorem}{Theorem}
\newtheorem{definition}{Definition}
\newtheorem{lemma}{Lemma}
\newtheorem{proposition}{Proposition}
\newtheorem{corollary}{Corollary}
\newtheorem{conjecture}{Conjecture}
\newtheorem{example}{Example}

\newcommand{\ZZ}{{\mathbb Z}}
\newcommand{\RR}{{\mathbb R}}
\newcommand{\CC}{\mathbb{C}}
\newcommand{\per}{\text{per }}

\def\Red{\textcolor{red}}

\definecolor{colO}{HTML}{DDCC77}
\definecolor{colB}{HTML}{6699CC}
\colorlet{col4}{PineGreen}
\colorlet{col5}{BurntOrange}

\tikzstyle{vertex}=[circle, draw, fill=black, inner sep=0pt, minimum size=4pt]
\tikzstyle{rvertex}=[circle, draw, fill=red, inner sep=0pt, minimum size=4pt]
\tikzstyle{smallvertex}=[circle, draw, fill=black, inner sep=0pt, minimum size=2pt]
\tikzstyle{edge4}=[very thick,draw=col4]
\tikzstyle{edge5}=[very thick,draw=col5]
\tikzstyle{edge}=[line width=1.5pt,black!50!white]
\tikzstyle{redge}=[line width=1.5pt,Red]
\tikzstyle{bledge}=[line width=1.5pt,black]
\tikzstyle{bedge}=[line width=1.5pt,NavyBlue,dashed]
\tikzstyle{ovedge}=[line width=1.5pt,OliveGreen]
\tikzstyle{edge41}=[line width=1.5pt, draw=YellowOrange]
\tikzstyle{cedge}=[line width=1.5pt, cyan]
\tikzstyle{lnode}=[circle,white,draw=black!80!white,fill=black!80!white,inner sep=0.5pt, font=\scriptsize]

\newcommand{\HIiid}{
	\begin{tikzpicture}[scale=0.75]
	\node[smallvertex] (1) at (0, 0) {};
	\node[smallvertex] (3) at (0.7, 0.3) {};      
	\draw[gray!60,densely dashed] (0.3,0.2) circle (0.6cm);
	\draw[very thick,->] (0.3,-0.5) -- (0.3,-0.9);
	\begin{scope}[yshift=-1.8cm]
	\node[smallvertex] (1) at (0, 0) {};
	\node[smallvertex] (3) at (0.7, 0.3) {}; 
	\node[smallvertex] (w) at (0.65, -0.55) {}; 
	\draw[gray!60,densely dashed] (0.3,0.2) circle (0.6cm);
	\draw[edge] (1) to (w);
	\draw[edge] (3) to (w);
	\end{scope}
	\node (u) at (1.0,0) {};
	\node (u) at (-0.4,0) {};
	\end{tikzpicture}
}

\newcommand{\HIiiid}{
	\begin{tikzpicture}[scale=0.75]
	\node[smallvertex] (1) at (0, 0) {};
	\node[smallvertex] (2) at (0.25, 0.4) {};
	\node[smallvertex] (3) at (0.7, 0.3) {};      
	\draw[gray!60,densely dashed] (0.3,0.2) circle (0.6cm);
	\draw[very thick,->] (0.3,-0.5) -- (0.3,-0.9);
	\begin{scope}[yshift=-1.8cm]
	\node[smallvertex] (1) at (0, 0) {};
	\node[smallvertex] (2) at (0.25, 0.4) {};
	\node[smallvertex] (3) at (0.7, 0.3) {}; 
	\node[smallvertex] (w) at (0.65, -0.55) {}; 
	\draw[gray!60,densely dashed] (0.3,0.2) circle (0.6cm);
	\draw[edge] (1) to (w);
	\draw[edge] (2) to (w);
	\draw[edge] (3) to (w);
	\end{scope}
	\node (u) at (1.0,0) {};
	\node (u) at (-0.4,0) {};
	\end{tikzpicture}
}

\newcommand{\HIIiid}{
	\begin{tikzpicture}[scale=0.75]
	\node[smallvertex] (1) at (0, 0) {};
	\node[smallvertex] (2) at (0.35, 0.5) {};
	\node[smallvertex] (3) at (0.7, 0.3) {};      
	\draw[redge] (1) to (2);
	\draw[gray!60,densely dashed] (0.3,0.2) circle (0.6cm);
	\draw[very thick,->] (0.3,-0.5) -- (0.3,-0.9);
	\begin{scope}[yshift=-1.8cm]
	\node[smallvertex] (1) at (0, 0) {};
	\node[smallvertex] (2) at (0.35, 0.5) {};
	\node[smallvertex] (3) at (0.7, 0.3) {}; 
	\node[smallvertex] (w) at (0.65, -0.55) {}; 
	\draw[gray!60,densely dashed] (0.3,0.2) circle (0.6cm);
	\draw[edge] (1) to (w);
	\draw[edge] (2) to (w);
	\draw[edge] (3) to (w);
	\end{scope}
	\node (u) at (1.0,0) {};
	\node (u) at (-0.4,0) {};
	\end{tikzpicture}
}

\newcommand{\HIIiiid}{
	\begin{tikzpicture}[scale=0.75]
	\node[smallvertex] (1) at (0, 0) {};
	\node[smallvertex] (2) at (0.35, 0.5) {};
	\node[smallvertex] (3) at (0.7, 0.3) {};      
	\node[smallvertex] (4) at (0.3,0.0) {};  
	\draw[redge] (1) to (2);
	\draw[gray!60,densely dashed] (0.3,0.2) circle (0.6cm);
	\draw[very thick,->] (0.3,-0.5) -- (0.3,-0.9);
	\begin{scope}[yshift=-1.8cm]
	\node[smallvertex] (1) at (0, 0) {};
	\node[smallvertex] (2) at (0.35, 0.5) {};
	\node[smallvertex] (3) at (0.7, 0.3) {}; 
	\node[smallvertex] (4) at (0.3,0.0) {}; 
	\node[smallvertex] (w) at (0.65, -0.55) {}; 
	\draw[gray!60,densely dashed] (0.3,0.2) circle (0.6cm);
	\draw[edge] (1) to (w);
	\draw[edge] (2) to (w);
	\draw[edge] (3) to (w);
	\draw[edge] (4) to (w);
	\end{scope}
	\node (u) at (1.0,0) {};
	\node (u) at (-0.4,0) {};
	\end{tikzpicture}
}

\newcommand{\HIIIx}{
	\begin{tikzpicture}[scale=0.75]
	\node[smallvertex] (1) at (-0.05, 0.1) {};
	\node[smallvertex] (2) at (0.35, 0.5) {};
	\node[smallvertex] (3) at (0.7, 0.3) {}; 
	\node[smallvertex] (4) at (0.1,-0.2) {}; 
	\node[smallvertex] (5) at (0.0,0.45) {}; 
	\draw[redge] (1) to (3);
	\draw[redge] (2) to (4);
	\draw[gray!60,densely dashed] (0.3,0.2) circle (0.6cm);
	\draw[very thick,->] (0.3,-0.5) -- (0.3,-0.9);
	\begin{scope}[yshift=-1.8cm]
	\node[smallvertex] (1) at (-0.05, 0.1) {};
	\node[smallvertex] (2) at (0.35, 0.5) {};
	\node[smallvertex] (3) at (0.7, 0.3) {}; 
	\node[smallvertex] (4) at (0.1,-0.2) {}; 
	\node[smallvertex] (5) at (0.0,0.45) {}; 
	\node[smallvertex] (w) at (0.7, -0.5) {}; 
	\draw[gray!60,densely dashed] (0.3,0.2) circle (0.6cm);
	\draw[edge] (1) to (w);
	\draw[edge] (2) to (w);
	\draw[edge] (3) to (w);
	\draw[edge] (4) to (w);
	\draw[edge] (5) to (w);
	\end{scope}
	\node (u) at (1.0,0) {};
	\node (u) at (-0.4,0) {};
	\end{tikzpicture}
}

\newcommand{\HIIIv}{
	\begin{tikzpicture}[scale=0.75]
	\node[smallvertex] (1) at (-0.05, 0.1) {};
	\node[smallvertex] (2) at (0.4, 0.55) {};
	\node[smallvertex] (3) at (0.7, 0.3) {}; 
	\node[smallvertex] (4) at (0.1,-0.2) {}; 
	\node[smallvertex] (5) at (0.0,0.45) {}; 
	\draw[redge] (1) to (2);
	\draw[redge] (2) to (4);
	\draw[gray!60,densely dashed] (0.3,0.2) circle (0.6cm);
	\draw[very thick,->] (0.3,-0.5) -- (0.65,-0.9);
	\begin{scope}[xshift=1.4cm]
	\node[smallvertex] (1) at (-0.05, 0.1) {};
	\node[smallvertex] (2) at (0.4, 0.55) {};
	\node[smallvertex] (3) at (0.7, 0.3) {}; 
	\node[smallvertex] (4) at (0.1,-0.2) {}; 
	\node[smallvertex] (5) at (0.0,0.45) {}; 
	\draw[redge] (1) to (5);
	\draw[redge] (5) to (3);
	\draw[gray!60,densely dashed] (0.3,0.2) circle (0.6cm);
	\draw[very thick,->] (0.3,-0.5) -- (-0.05,-0.9);
	\end{scope}
	\begin{scope}[yshift=-1.8cm, xshift=0.7cm]
	\node[smallvertex] (1) at (-0.05, 0.1) {};
	\node[smallvertex] (2) at (0.4, 0.55) {};
	\node[smallvertex] (3) at (0.7, 0.3) {}; 
	\node[smallvertex] (4) at (0.1,-0.2) {}; 
	\node[smallvertex] (5) at (0.0,0.45) {}; 
	\node[smallvertex] (w) at (0.7, -0.5) {}; 
	\draw[gray!60,densely dashed] (0.3,0.2) circle (0.6cm);
	\draw[edge] (1) to (w);
	\draw[edge] (2) to (w);
	\draw[edge] (3) to (w);
	\draw[edge] (4) to (w);
	\draw[edge] (5) to (w);
	\end{scope}
	\end{tikzpicture}
}

\newcommand{\Des}{
	\begin{tikzpicture}[scale=0.8]
	\node[vertex,label=right:6] (0) at (5, 3) {1};
	\node[vertex,label=left:3] (1) at (0, 0) {2};
	\node[vertex,label=above: 2] (2) at (1.5, 1.5) {3};
	\node[vertex,label=above:4] (3) at (3.5, 1.5) {4};
	\node[vertex,label=right:5] (4) at (5, 0) {5}; 
	\node[vertex,label=left:1] (5) at (0, 3) {6}; 
	
	\draw[bedge] (2) to (5);
	
	\draw[edge] (0) to (5);
	\draw[edge] (1) to (5);
	\draw[edge] (2) to (3);
	\draw[edge] (1) to (2);
	\draw[edge] (1) to (4);
	\draw[edge] (0) to (4);
	\draw[edge] (0) to (3);
	\draw[edge] (4) to (3);
	\end{tikzpicture}
}

\newcommand{\Jackson}{
	\begin{tikzpicture}[scale=0.8]
	\node[vertex,label=left:6] (0) at (-1, 2) {};
	\node[vertex,label=above:8] (1) at (3, 3) {};
	\node[vertex,label=below:3] (2) at (2, -2) {};
	\node[vertex,label=right:4] (3) at (3, 0) {};
	\node[vertex,label=below:2] (4) at (-1, -2) {}; 
	\node[vertex,label=above:5] (5) at (0,3) {};
	\node[vertex,label=left:1] (6) at (0, 0) {};
	\node[vertex,label=right:7] (7) at (2, 2) {}; 
	
	\draw[bedge] (6) to (4);
	
	\draw[edge] (0) to (4);
	\draw[edge] (0) to (5);
	\draw[edge] (0) to (7);
	\draw[edge] (1) to (3);
	\draw[edge] (1) to (5);
	\draw[edge] (1) to (7);
	\draw[edge] (2) to (3);
	\draw[edge] (2) to (4);
	\draw[edge] (2) to (7);
	\draw[edge] (6) to (3);
	\draw[edge] (6) to (5);
	\draw[edge] (6) to (7);
	\end{tikzpicture}
}

\newcommand{\Lmaxeight}{
	\begin{tikzpicture}[scale=0.75]
	
	\node[vertex, label=left:1] (1) at (-2.8, -1){1};
	\node[vertex, label=left:2] (2) at (-2.8, 3) {2};
	\node[vertex, label=right:3] (3) at (5.2, 3) {3};
	\node[vertex, label=right:4] (4) at (5.2, -1) {4};
	\node[vertex, label=below:5] (5) at (4.8,1.7) {5};
	\node[vertex,label=below:6] (6) at (3.6,2.75*0.45)  {6};
	\node[vertex,label=left:7] (7) at (-1.1,2.75*0.5) {7};
	\node[vertex, label=above:8] (8) at (-2.3,0.25) {8};
	\draw[bedge] (1) to (2);
	\draw[edge]   (2)edge(3) (3)edge(4) (1)edge(4) (6)edge(7) (6)edge(4);
	\draw[edge] (5)edge(6) (5)edge(2) (5)edge(3) (1)edge(8) (2)edge(7);
	\draw[edge] (8)edge(7) (8)edge(4);
	\end{tikzpicture} \hspace*{2mm}
}

\title{On the multihomogeneous B\'ezout bound on the number of embeddings of minimally rigid graphs}

\author[1,2]{Evangelos Bartzos}
\author[1,2]{Ioannis Z. Emiris}
\author[3]{Josef Schicho}

\affil[1]{Department of Informatics and Telecommunications, National Kapodistrian University of Athens}

\affil[2]{ATHENA Research Center}
\affil[3]{Research Institute for Symbolic Computation,  Johannes Kepler University, Linz}

\date{}

\begin{document}
	\maketitle
\begin{abstract}
  Rigid graph theory is an active area with many open problems, especially regarding embeddings in $\mathbb{R}^d$ or other manifolds, and tight upper bounds on their number for a given number of vertices.
  Our premise is to relate the number of embeddings to that of solutions of a well-constrained algebraic system and exploit progress in the latter domain.
  In particular, the system's complex solutions naturally extend the notion of real embeddings, thus allowing us to employ bounds on complex roots.
  
  We focus on multihomogeneous B{\'e}zout (m-B{\'e}zout) bounds of algebraic systems since they are fast to compute and rather tight for systems exhibiting structure as in our case. 
  We introduce two methods to relate such bounds to combinatorial properties of minimally rigid graphs in $\CC^d$ and $S^d$.
  The first relates the number of graph orientations to the m-B\'ezout bound, while the second leverages a matrix permanent formulation.
  Using these approaches we improve the best known asymptotic upper bounds for planar graphs in dimension~3, and all minimally rigid graphs in dimension $d\geq 5$, both in the Euclidean and spherical case.
  
  Our computations indicate that m-B\'ezout bounds are tight for embeddings of planar graphs in $S^2$ and $\CC^3$.
  We exploit Bernstein's second theorem on the exactness of mixed volume, and relate it to the m-B{\'e}zout bound by analyzing the associated Newton polytopes.
  We reduce the number of checks required to verify exactness by an exponential factor, and conjecture further that it suffices to check a linear instead of an exponential number of cases overall.
  
\end{abstract}

\section{Introduction}

Let $G=(V,E)$ be a simple undirected graph, with vertices $V$ and edges $E\subset~V\times V$, and let $\mathbf{p} = \{ p_1, p_2, \dots p_n \} \in \RR^{d n}$ be a set of vectors $p_u\in\RR^d$ assigned to the vertices, where $n=|V|$ is the vertex cardinality.
This assignment is a (real) \textit{embedding} of $G$ in $\RR^d$. 
Every embedding induces a set of non-negative real edge lengths $\bm{\lambda}=  (\lVert p_u-p_v \rVert)_{(u,v) \in E}$,
where $\lVert \cdot \rVert$ denotes the Euclidean distance.
The distances $\bm{\lambda}$ represent a labelling of the graph edges; hence such graphs are also known as distance graphs. 
The embedding $\mathbf{p}$ of $G$ is called \textit{rigid} if it admits only a finite number of other embeddings up to rigid motions, for the edge labelling induced by $\mathbf{p}$, otherwise it is called \textit{flexible}.

This basic dichotomy can be associated to the graph for almost all embeddings with no reference to a specific embedding.
This is achieved using the following genericity notion.
A graph embedding is called \textit{generic} if a small perturbation of the embedded vertices does not alter whether the embedding is rigid or flexible \cite{GSS93}.
Thus, $G$ is \textit{generically rigid} in $\RR^{d}$ if it is rigid for every edge labelling induced by a generic embedding.
Embeddings of simple graphs are also defined in $\CC^d$ and correspond to all possible configurations of $n$ points $\bm{x}\in\CC^{d n}$ that satisfy the assignments 
$$
\lambda_{u,v}^2= \sum\limits_{k=1}^d (x_{u,k}-x_{v,k})^2,
$$
for all edges $(u,v)\in E$.
Note that this does not represent a norm in $\CC^d$.

Additionally, a rigid graph  is (generically) {\em minimally} rigid if any edge removal breaks the rigidity.
A fundamental theorem in graph rigidity due to Maxwell gives a necessary (but not sufficient) condition on the edge count of a graph and all its subgraphs for the graph to be rigid \cite{Maxwell}.
More precisely, if a graph $G=(V,E)$ is minimally rigid in $\RR^d$, then $|E|=d \cdot |V| -\binom{d+1}{2}$ and $|E'| \leq d\cdot |V'|-\binom{d+1}{2}$ for every subgraph $G'(V',E') \subset G$.

Minimally rigid graphs in $\RR^2$ are known as \textit{Laman graphs}, because of G.~Laman's theorem that gives a full characterisation of them \cite{Laman}. 
Laman's theorem states that Maxwell's condition is also sufficient for minimally rigid graphs in $\RR^2$.
This condition is not the only one that verifies minimal rigidity in $\RR^2$.
For example, Henneberg constructions, pebble games and Recski's theorem have been used alternatively \cite{handbook1,pebble}.
In fact, Laman rediscovered the forgotten results of H.\ Pollaczek-Geiringer \cite{Geiringer1927,Geiringer1932} and in order to honour her we call minimally rigid graphs in $\RR^3$ \textit{Geiringer graphs} following \cite{belt,GKT17}.
In dimension~3, a well-known theorem by Cauchy states that all strictly convex simplicial polyhedra are rigid \cite{Whit84}.
The edge-skeleta of these polyhedra constitute the subclass of planar (in the graph-theoretical sense) Geiringer graphs.
On the other hand,  there is no full characterization for the whole class of Geiringer graphs, since Maxwell's condition is not always sufficient for minimal rigidity in dimension 3 or higher (see Figure~\ref{fig:doubleb} for a counter-example).
\begin{figure}	\begin{center}
		\includegraphics[width=0.4\textwidth]{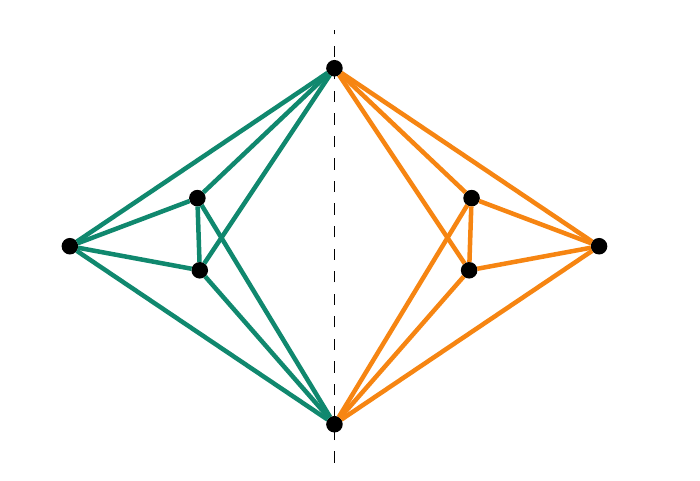}
		\caption{The double-banana graph satisfies Maxwell's condition in $\RR^3$, but is not rigid since its two rigid components revolve around the dashed axis.\label{fig:doubleb}}
\end{center}\end{figure}

Rigid graphs in $\RR^2$ and $\RR^3$ have attracted the principal interest of research because of their numerous applications.
Such applications can be found in robotics \cite{Rob2,Drone}, 
molecular biology \cite{Bio1,Bio2}, sensor network localization \cite{sensor} and architecture \cite{arch1,arch2}.
Beyond these applications, rigidity theory has also a mathematical interest on its own and 
results on rigidity can be extended to arbitrary dimension and to other manifolds \cite{handbook1,tay,NOP2012}.
More specifically, the set of minimally rigid graphs on every unit $d$-dimensional sphere $S^d$, coincides with the set of minimally rigid graphs in $\RR^d$ \cite{Whiteley_cone}.

The main contribution of this paper is to employ a variety of algorithms from algebraic geometry in order to investigate minimally rigid graphs, relying on the observation that they can be adequately modeled by well-constrained algebraic systems.
The total number of edge constraints according to Maxwell's theorem equals the total number of degrees of freedom (dof) of the rigid graph embedding.
They equal the total number of vertex coordinates, namely $d\cdot n$, after subtracting the $\binom{d+1}{2}$ dof of rigid motions (rotations and translations).
Besides, the inequalities in Maxwell's condition exclude any over-constrained subsystem.

We shall concentrate on complex embeddings.
Note that the number of complex embeddings of a minimally rigid graph $G$ is the same for all generic assignments of lengths $\bm{\lambda}$ 
(see~\cite{Jackson} for Laman graphs and~\cite{SomWam} on coefficient-parameter theory).
This number is denoted by $c_d(G)$ and is bounded by upper bounds on the number of (complex) solutions of the algebraic system.

\paragraph{Previous work:} 
A major open question in rigidity theory is to obtain optimal asymptotic upper bounds on the number of embeddings (see for example \cite{Jackson}).
According to Maxwell's condition, there are $O(d n)$ edge constraints that can be represented by the quadratic Euclidean edge length equations $\displaystyle\lVert \mathbf{x}_u-\mathbf{x}_v \rVert^2~=~\lambda_{u,v}^2$.
Applying B\'ezout's bound, this formulation yields $O(2^{dn})$ as an upper bound on the number of embeddings.

In \cite{Steffens} the authors used mixed volume techniques on a set of equations that reformulates the quadratic edge length constraint equations to compute better BKK bounds (named after Bernstein, Khovanskii and Kushnirenko) on the number of embeddings for Laman graphs.
We will also use this formulation, however, their polyhedral methods did not improve B\'ezout's asymptotic bound in the general case.

Besides edge lengths equations, there are distance geometry methods to verify embeddings in $\RR^d$, subject to determinantal equations and inequalities obtained from Cayley-Menger matrices \cite{Blu}.
The best known upper bound for the number of embeddings of a rigid graph in dimension $d$ uses these determinantal varieties, applying a theorem on their degree \cite{Borcea} (for the degree of determinantal varieties, see~\cite{HarrisTu}):
$$
\prod\limits_{k=0}^{n-d-2} \frac{\binom{n-1+k}{n-d-1-k}}{\binom{2k+1}{k}} .
$$
This bound does not improve upon B\'ezout's bound asymptotically.

Let us now juxtapose these to lower bounds.
Direct computations for Laman graphs have proved that there are graphs with $2.50798^n$ complex embeddings in $\CC^2$ and Geiringer graphs with $3.06825^n$ complex embeddings \cite{GKT17}.
On the sphere it is proven that there are graphs with $2.5698^n$ complex embeddings \cite{count_sphere}.
Computations on the number of embeddings for all graphs with given $n$ have been completed only for relatively small $n$ due to the amount of computation required (up to $n=13$ in dimension~2 and $n=10$ in dimension~3).
In any case, the gap between asymptotic upper bounds and experimental results remains enormous.

Finding the exact number of complex solutions requires some demanding computations.
In the case of embeddings of Laman graphs both in $\CC^2$ and $S^2$, there are combinatorial algorithms \cite{Joseph_lam,count_sphere} that speed up computations a lot, but still it is almost infeasible to compute $c_d(G)$ for graphs with more than 18 vertices in a desktop computer.
Since no similar algorithm exists for Geiringer graphs, solving algebraic systems for random instances is the only option:
Groebner bases~\cite{GKT17} and homotopy solvers like \texttt{phcpy} \cite{phcpy,belt} have been employed.
For the 11-vertex case, \texttt{phcpy} fails to give all solutions for many graphs, while Groebner bases may take more then 3 days for a single graph (if the algorithm terminates).
Another homotopy solver that has come recently to our attention is \texttt{MonodromySolver} \cite{monodromySolver}, implemented in \texttt{Macaulay2}.
We tested this solver for a variety of graphs and it seems more accurate and considerably faster than \texttt{phcpy}.

Besides computing the exact number of embeddings, complex bounds have been considered as general estimates of the number of solutions.
These bounds can also serve as input to homotopy continuation solvers.
Mixed volume computation has been used, by applying suitable variable transformations \cite{etv,belt}.
Even if mixed volume computations are generally faster than exact estimations, the computational limits of this method are also rather restrained.

Let us comment that there are real algebraic bounds \cite{BihSot,Khovanskii} which are sharper than the complex ones for certain polynomial structures.
These bounds are much higher than the (more general) complex bounds already mentioned in the case of minimally rigid graphs.

\paragraph{Our contribution.} 
In the present paper we use m-B\'ezout bounds on the number of complex embeddings.
We present a new recursive combinatorial algorithm that adopts a graph-theoretic approach in order to speed up the computation of m-B\'ezout bounds in our case, based on a standard partition of variables.
We also use matrix permanent computation, which are known to compute m-B\'ezout bounds \cite{emvid}, comparing runtimes with our algorithm.
Applying the best known upper bounds for orientations of planar graphs~\cite{Felsner} and permanents of $(0,1)$-matrices (known as Br\`egman-Minc bound \cite{Bre73,Minc63}), we improve asymptotic upper bounds for planar minimally rigid graphs in dimension~3, and for all minimally rigid graphs in dimensions $d\geq 5$.

We compare the m-B\'ezout bounds with mixed volume bounds and the actual number of complex embeddings of all Laman and  Geiringer graphs with $n\leq 9$ vertices, and some selected Laman graphs up to $n=18$ and Geiringer graphs up to $n=12$, and observe that m-B\'ezout is exact for  the large majority of spherical embeddings in the case of planar  Laman graphs, while it is exact for all planar Geiringer graphs.
We adjust the Bernstein's discriminant conditions on the exactness of mixed volume to the case of m-B\'ezout bounds using Newton polytopes whose mixed volume equals to the m-B\'ezout.
Our method exploits the system structure to ensure that some conditions are verified {\em \`a priori}, reducing the number of checks required by an exponential factor.
This number of conditions remains exponential, but based on our experiments
we conjecture that it suffices to check only a linear number of cases overall.
These results are highlighted by certain examples.

The rest of the paper is structured as follows.
The next section discusses some background and introduces the algebraic formulation that we exploit.
Section~\ref{sec:mBez} offers two approaches for efficiently computing m-B\'ezout bounds: first a combinatorial method and, second, a reduction to the permanent thus improving bounds for planar graphs in dimension $3$ and for all graphs in the case of $d\geq 5$.
Section~\ref{sec:exact} studies the exactness of m-B\'ezout bounds.
The paper concludes with open questions.

\section{Henneberg constructions and algebraic formulation}\label{sec:eq-sphere}

In this section we present some preliminaries about the construction of rigid graphs and the computation of their embeddings.
In general, Maxwell's condition is not suitable to find the set (or a superset) of all minimally rigid graphs with a given number of vertices.
On the other hand, a sequence of moves known as Henneberg steps can construct such sets.
Additionally, a characterization of minimally rigid graphs up to the last Henneberg move can be used to separate cases that have a trivial number from the the non-trivial ones.
Subsequently, we present an algebraic formulation that is based on a variant of the quadratic edge lengths equations.
This formulation has been used several times in the context of studies on rigid graphs that exploit sparse elimination techniques \cite{Steffens,etv,belt}.

\subsection{Henneberg steps}\label{sec:Hsteps}
\begin{figure}[htb!]
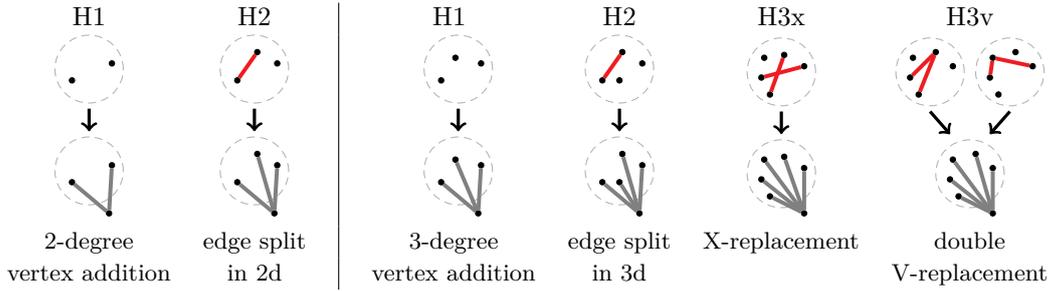

	\centering
	\hspace*{-6mm}
	\begin{tabular}{cc|cccc}
		H1  & H2 \hspace*{1mm}& \hspace*{1mm}H1  & H2 & H3x  & H3v\\
		\HIiid	& \HIIiid \hspace*{1mm} & \hspace*{1mm} \HIiiid	& \HIIiiid & \HIIIx & \HIIIv  \\
		\small{2-degree} & \small{edge split} \hspace*{1mm} & \hspace*{1mm} \small{3-degree}& \small{edge split} & \small{X-replacement} & \small{double}\\	
		\small{vertex addition}& \small{in 2d} \hspace*{1mm} & \hspace*{1mm} \small{vertex addition}& \small{in 3d} & & \small{V-replacement} \\	
	\end{tabular} 
	\caption{Henneberg steps in dimensions 2 \& 3.}
\end{figure}
Minimally rigid graphs in $\RR^d$ can be constructed as a sequence of Henneberg moves starting from the complete graph on $d$ vertices $K_d$.
In the case of $d=2$ all minimally rigid graphs can be obtained by Henneberg 1 (H1) and Henneberg 2 (H2) operations, giving one more method to characterize Laman graphs.
On the other hand these two moves are not sufficient to construct all the minimally rigid graphs in $d=3$, so an extended Henneberg step is required.
These 3 moves give a superset of Geiringer graphs.
It is conjectured that H1, H2 and H3 completely characterize rigid graphs in $\RR^3$ \cite{handbook1,tay}.

These moves generalize to arbitrary dimension.
The H1 move in dimension $d$ adds a new $d$-degree vertex, while H2 adds a $(d+1)$-degree vertex removing also an edge.
It has been proven that these moves always preserve rigidity in all dimensions.
On the other hand, H3 step in dimension~4 is not always rigid~\cite{H3contre}.

We used Henneberg steps to construct sets of Laman and Geiringer graphs up to isomorphism, using canonical labeling as in \cite{Joseph_lam,GKT17}.
Since Henneberg moves add a vertex with a fixed degree, we can separate the sets of graphs with the same number of vertices up to their minimal degree.
So if a graph in dimension $2$ has minimal degree $2$, then it can be constructed with an H1 move in the last step.
If the minimal degree is $3$ it certainly requires an H2 move in the last step of the Henneberg sequence.
Notice that the H1 move trivially doubles the number of embeddings, since the new vertex lies in the intersection of $d$ different $(d-1)$-spheres.
This means that we may examine only graphs that are constructed with the other Henneberg moves.
Let us comment that the Geiringer graphs whose construction requires an H3 move have minimal degree $5$ and no such graph exist for any graph with $n\leq 11$ vertices.

\begin{table}
	\caption{Numbers of Laman and Geiringer graphs up to the last Henneberg move and graph planarity.}
	\centering
	\begin{tabular}{|c||c|c|c|c|c|c|c|c|c|}
		\multicolumn{10}{c}{\textbf{Laman graphs}} \vspace{1.5mm}\\
		\hline
		$\bm{n}$  & $\bm{3}$ & $\bm{4}$ & $\bm{5}$ & $\bm{6}$ & $\bm{7}$ & $\bm{8}$ & $\bm{9}$ & $\bm{10}$ & $\bm{11}$ \\
		\hline
		H1 planar & $1$ & $1$ & $3$ & $11$ &  $62$ & $491$ & $5,041$ & $60,040$ & $791,195$ \\
		H1 non-planar & - & - & - & - &  $4$ & $85$ & $1,917$ & $46,903$ & $1,201,401$   \\
		\hline
		H2 planar & - & - & - & $1$ &  $3$ & $18$ & $122$ & $1,037$ & $9,884$ \\
		H2 non-planar & - & - & - & $1$ &  $1$ & $14$ & $142$ & $2,152$ & $36,793$  \\
		\hline
		\multicolumn{10}{c}{  }
	\end{tabular}
	
	\begin{tabular}{|c||c|c|c|c|c|c|c|c|}
		\multicolumn{9}{c}{\textbf{Geiringer graphs}} \vspace{1.5mm}\\
		\hline
		$\bm{n}$  & $\bm{4}$ & $\bm{5}$ & $\bm{6}$ & $\bm{7}$ & $\bm{8}$ & $\bm{9}$ & $\bm{10}$ & $\bm{11}$ \\
		\hline
		H1 planar &  $1$ & $1$ & $1$ &  $4$ & $12$ & $45$ & $221$ & $1,215$ \\
		H1 non-planar &  - & - & $2$ &  $16$ & $ 299 $ & $9,718$ & $527,250$ & $41,907,790$   \\
		\hline
		H2 planar &  - & - & $1$ &  $1$ & $2$ & $5$ & $12$ &  $34$  \\
		H2 non-planar &  - & - & - & $5$ &  $61$ &  $1,719$ & $ 85,401 $ & $6,267,144$  \\
		\hline
	\end{tabular}
	\label{tab:Hsteps}
\end{table}

\subsection{Sphere equations}
In order to compute the number of embeddings of a rigid graph we used some standard algebraic formulation  \cite{Steffens,etv}.
First we remove rigid motions by fixing $\binom{d+1}{2}$ coordinates yielding a $0-$dimensional system.
In the case of dimension $2$, we may fix both coordinates of one vertex and one coordinate of a second vertex.
If these vertices are adjacent to one edge, then the length constraint imposes only one solution for the remaining coordinate of the second vertex up to rotations.
In general, if the graph contains a complete subgraph with $d$ vertices $v_1,v_2,\dots v_d $, then we can choose the coordinates of this $K_d$ graph in a way that they satisfy the edge lengths of this subgraph.
So, in the case of Laman graphs, we need to fix an edge, while in 3 dimensions a triangle should be fixed.
Note that for the first set of graphs there is always a $K_2$ (edge).
As for the 3-dimensional case, Geiringer graphs with no triangles ($K_3$) are very rare (the first one is the 10-vertex $K_{6,4}$).
In that situation, the vertices of an edge are fixed, while a third vertex is located on the same fixed plane as the edge, leaving two degrees of freedom of this vertex free.
In that way, every embedding is counted twice, so the number of solutions is divided by 2.

We use two types of equations in our systems.
The first set of equations at hand are the \textit{edge equations}, which represent the edge length constraints between the adjacent vertices of an edge.
Although these equations suffice to find the embeddings of a graph, we cannot take advantage of their structure and compute efficient bounds.
To overcome this problem, we define the set of \textit{magnitude equations} that introduce new variables representing the distance of each vertex from the origin.
Following \cite{belt}, we call the combination of these two sets of equations \textit{sphere equations}.

\begin{definition}
	\label{def:sphereEquations}
	Let $G=(V,E)$ be a graph with $n$ vertices.
	We denote by $\bm{\lambda}=(\lambda_{u,v})_{(u,v)\in E}\in \RR_+^{|E|}$ the (given) edge lengths and by $\widetilde{X}_u=\{x_{u,1},x_{u,2},\dots x_{u,d}\}$ the variables assigned to the coordinates of each vertex.
	The following system of equations gives the embeddings of $G$:
	\begin{equation}  \label{sphere_eq}	\begin{split}
	\lVert \widetilde{X}_u \rVert^2 &= s_u , \qquad\, \forall\, u \in V\,, \\
	s_u +s_v -2\langle \widetilde{X}_u,\widetilde{X}_v \rangle &= \lambda_{u,v}^2, \qquad \forall\, (u,v) \in E\backslash edges(K_d)\,,
	\end{split}
	\end{equation}
	where $\langle \widetilde{X}_u,\widetilde{X}_v \rangle$ is the Euclidean inner product.
	We will denote the set of variables $X_u=\widetilde{X}_u \bigcup \{s_u\}$ in the euclidean case using $s_u$ as the $(d+1)$-th variable $x_{u,d+1}$.
	If a vertex is fixed, its variables are substituted with constant values.
	This formulation can be obviously used in the case of embeddings on the unit $d$-dimensional sphere $S^d$ using $|\widetilde{X}_u|=d+1$ coordinates and setting $s_u=1$.
\end{definition}

This algebraic system has $m=d\cdot n-d^2$ edge equations and $n-d$ magnitude equations if there is at least one subgraph $K_d$  of $G$. 
Notice that the edges of the fixed $K_d$ serve to specify the fixed vertices and are not included in this set of equations, so $m < |E|$.
We remark that the m-B\'ezout bound (or the mixed volume bound) of a graph $G$ may vary up to the choice of the fixed $K_d$, so one needs to compute m-B\'ezout bounds up to all fixed $K_d$'s in order to find the minimal one.
On the other hand $c_d(G)$ is invariant under different choices of fixed $K_d$.

\section{Computing m-B\'ezout bounds}\label{sec:mBez}
In this section we propose two methods for computing m-B\'ezout bounds for minimally rigid graphs.
First, we give definitions of two classical complex bounds (m-B\'ezout, BKK) and we propose a natural partition in the set of variables for the m-B\'ezout.
Subsequently, we introduce an algorithm based on a relation between indegree constrained graph orientations and the m-B\'ezout bound of minimally rigid graphs.
Besides that, we also give an alternative way to compute this bound via the matrix permanent.
Let us mention that matrix permanents have been already used to bound the number of Eulerian orientations (which are graph orientations with equal indegree and outdegree for every vertex) in~\cite{Schrijver1983}, but to the best of our knowledge there are no published results on the connection between matrix permanents and indegree constrained graph orientations in the general case.
Consequently, we use an existing bound on the orientations of planar graphs to improve the asymptotic upper bounds for planar Geiringer graphs.
We also improve asymptotic upper bounds for $d\geq5$ applying a bound on the permanents of $(0,1)$-matrices.
Finally, we compare our combinatorial algorithm with existing methods to compute upper bounds or the actual number of embeddings in the cases of $\CC^2$ and $\CC^3$.

\subsection{The m-B\'ezout bound of the sphere equations}\label{sec:SphmBez}

Let us start with defining the bound in question.
It is based on a classical theorem on bounding the number of solutions of 0-dimensional well-constrained (or square) algebraic systems, see e.g.\ \cite{Shafarevich2013}:

\begin{theorem}
	\label{thm:mBezout}
	Let $X_1,X_2,\dots,X_k$ be a partition of the $m$ variables of a square algebraic system, such that $m_i=|X_i|$ and $m=m_1+m_2+ \dots + m_k$.
	Let $\alpha_{i,j}$ be the degree of the i-th equation in the j-th set of variables.
	Then, if the system has a finite number of solutions, it is bounded from above by the coefficient of the monomial $X_1^{m_1}\cdot X_2^{m_2} \cdots X_k^{m_k}$ in the polynomial 
	\begin{equation}
	\prod_{i=1}^{m} (\alpha_{i,1}\cdot X_1+\alpha_{i,2} \cdot X_2+ \dots + \alpha_{i,k} \cdot X_k).
	\label{eq:mbez}
	\end{equation}
\end{theorem}
This so-called m-B{\'e}zout bound improves considerably the classical B{\'e}zout bound, if the variables can be grouped in an appropriate way, exploiting the structure of the polynomial system.\\

The structure of polynomials can be exploited further using sparse elimination techniques.
One of the basic tools in sparse elimination theory is the notion of Newton Polytopes of a polynomial system.
Given
$$
f= \sum \limits_{\alpha\in\ZZ^m} c_{\alpha} \mathbf{x}^{\alpha} \in K[x_1,x_2, \dots ,  x_m ],\quad \mathbf{x}=(x_1,\dots,x_m),\; c_\alpha\ne 0,
$$
the Newton polytope $NP(f)$ is defined  as the convex hull of its monomial exponent vectors $\alpha$.
A very important theorem relates the structure of these polytopes to the number of solutions of algebraic systems in the corresponding toric variety.
The toric variety is a projective variety defined essentially by the Newton polytopes of the given system and contains the topological torus $(\CC^*)^m$ as a dense subset.
The set-theoretic difference of a toric variety and $(\CC^*)^m$ is toric infinity in correspondence with projective infinity.
In practical applications, as in this paper, one is interested in roots in the torus $(\CC^*)^m$, hence all toric roots except from those lying at toric infinity, in other words only in affine toric roots.

\begin{theorem}[BKK bound \cite{Bernshtein1975,Kouchnirenko1976,Khovanskii1978}]\label{thm:BKK}
	Let $ (f_i)_{1 \leq i \leq m } $ be a square 
	system of equations in
	$\CC[x_1,x_2, \dots ,  x_m ]$, and let $(NP(f_i))_{1\leq i\leq m}$ be their Newton Polytopes.
	Then, if the number of system's solutions in $(\CC^*)^m$ is finite, it is bounded above by the mixed volume of these Newton Polytopes.
\end{theorem}

Roughly, without paying much attention to the underlying variety, we have the following relations as in~\cite{SomWam}:
$$
\# \text{complex solutions} \leq \text{ mixed volume} \leq \text{ m-B{\'e}zout} \leq \text{ B{\'e}zout}. $$
On the other hand, the complexity of computing bounds goes in the opposite direction.

Here, we concentrate on the m-B\'ezout bound of the sphere equations of a graph $G=(V,E)$ up to a fixed complete subgraph $K_d$.
For the rest of the text, unless further specified, $K_d$ will denote a given complete subgraph and not all possible choices.
In general finding the optimal multihomogeneous partition is not in APX \footnote{APX is the class of all NP optimization problems, for which there exist polynomial-time approximation algorithms.
	This approximation is bounded by a constant (for further details see \cite{APX}).}, unless P=NP \cite{Malajovich2007}.
Here we choose a natural  partition such that each subset of variables contains these ones which correspond to the coordinates and the magnitude of a single vertex $X_u$.
In order to compute the m-B{\'e}zout bound, we will separate the magnitude equations from the edge equations.
In the first ones, there is only one set of variables with degree 2, while in every edge equation the degree of the $k$-th set of variables is always~1:
\begin{equation*}  \label{bez_sphere}
\begin{split}
\prod_{u \in V'} 2\cdot X_u \prod_{(u,v) \in E'} (X_{u}+X_{v}) =
2^{n-d} \cdot \prod_{u \in V'} X_u \prod_{(u,v) \in E'} (X_{u}+X_{v}),
\end{split}
\end{equation*}
where $E'=E\backslash edges(K_d)$, $V'=V\backslash vertices(K_d)$   and $X_{u} \equiv 0 $ if $u \in K_d$.

This means that  we only need to find the coefficient of the monomial $ \displaystyle\prod\limits_{u \in V'} X_u^d$ in the polynomial of the product:
\begin{equation} \label{eq:graph2mBezout}
\prod_{(u,v) \in E'} (X_{u}+X_{v})
\end{equation}
Let us denote this coefficient by $mB_{E}(G,K_d)$, which is related only to the combinatorial structure of edge equations.
The m-B{\'e}zout bound for the number of embeddings of a graph in $\CC^d$ up to a fixed $K_d$ is $mB(G,K_d)=2^{n-d} \cdot mB_{E}(G,K_d)$.
Notice that this bound is the same for spherical embeddings in $S^d$.

\subsection{A combinatorial algorithm to compute m-B{\'e}zout bounds}\label{sec:comb}

This subsection focuses on m-B{\'e}zout bounds for minimally rigid graphs.
Our method is inspired by two different approaches that characterize Laman graphs.
First, Recski's theorem states that if a graph is Laman then any multigraph obtained by doubling an edge should be the union of two spanning trees \cite{handbook1}.
Additionally, pebble games give a relation between the existence of an orientation and the number of constraints in a graph and its subgraphs \cite{pebble}.
The following theorem gives a combinatorial method to compute the m-B{\'e}zout bound, proving that $mB_{E}(G,K_d)$ is exactly the number of certain indegree-constrained orientations.

\begin{theorem}
	Let $H_{K_d}(V,E')$ be a graph obtained after removing the edges of $K_d$ from $G=(V,E)$.
	We define $\mathcal{H}_{K_d}$ as the set of all orientation of $H_{K_d}$ such that each non-fixed vertex has indegree $d$ and each fixed vertex has indegree $0$.
	Then, the coefficient of the monomial $ \displaystyle\prod\limits_{u \in V'} X_u^d$ of the product in expression~(\ref{eq:graph2mBezout}) is exactly the same as $|\mathcal{H}_{K_d}|$.
\end{theorem}

\begin{proof}
	By expanding the product $\prod\limits_{(u,v) \in E'} (X_{u}+X_{v})$, the monomial $ \displaystyle\prod\limits_{u \in V'} X_u^d$ can be obtained when each $X_{u}$ from a given edge  contributes exactly $d$ times in that product.
	This means that every time we shall choose one of the two sets of variables that correspond to the adjacent vertices  of the edge represented in the parenthesis.
	This choice yields an orientation in the directed graph and vice versa.
	Thus, the number of different orientations in all edges gives us how many times this monomial will appear in the expansion, completing the proof.
\end{proof}

This theorem gives another way to prove that an H1 move doubles the m-B\'ezout bound of minimally rigid graphs.
Hence, minimally rigid graphs constructed only by H1 moves have at most $2^{n-d}$ embeddings (actually this bound is tight, see Section~\ref{sec:Hsteps}).

\begin{corollary}\label{cor:orientH1}  
	An H1 move always doubles the m-B{\'e}zout bound up to the same fixed $K_d$.
\end{corollary}
\begin{proof}
	Let $|\mathcal{H}_{K_d}|$ be the number of indegree-constrained orientations  for a graph $G$ with $n$ vertices up to a given $K_d$.
	This means that the m-B\'ezout bound is  
	$$
	mB(G,K_d)=2^{n-d}\cdot |\mathcal{H}_{K_d}| .
	$$
	Now, let $G^*$ be a graph obtained by an H1-move on the graph $G^*$.
	Since H1 adds a degree-$d$ vertex to $G$, this means that there is only one way to reach indegree $d$ for the new vertex of $G^*$.
	So the indegree-constrained orientations of $G^*$ up to the same $K_d$ are exactly $|\mathcal{H}_{K_d}|$ and 
	$$
	mB(G^*,K_d)=2^{n+1-d}\cdot |\mathcal{H}_{K_d}|=2\cdot mB(G,K_d).
	$$
	
	The proof concludes by induction: Starting from $K_d$, the m-B\'ezout bound of a minimally graph constructed only by H1 moves is $2^{n-d}$.
\end{proof}

\begin{figure}[htp]
	\begin{center}
		\includegraphics[width=0.7\textwidth]{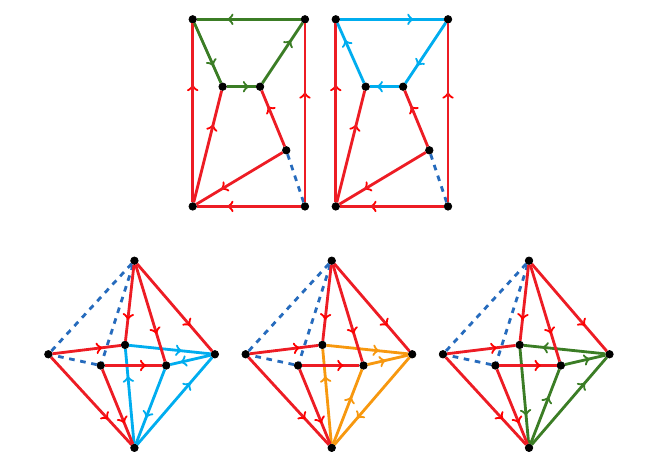}
		\caption{The orientations of graphs $L56$ and $G48$. 
			Notice that there is only one way to direct the red edges up to the choice of $K_d$ (dashed blue).
			\label{fig:orient}}
	\end{center} 
\end{figure}

\begin{algorithm}[htp!]\label{alg:orient}
	\caption{Count graph orientations}
	\DontPrintSemicolon
	\KwFn{orient}\\
	\KwInput{$n$ (\# of vertices), $E$ (graph edges $\backslash K_d$),  $indeg$ (desired indegree list. \\ If vertex $u$ fixed then $indeg[u]=0$, otherwise $indeg[u]=d$),  }
	\KwOutput{\# of indegree-constrained orientations}
	deg = vertex degrees of graph $G([n],E)$ \\
	\tcc{Ending condition for the recursion}
	\If{$|E|=0$ }
	{
		\Return(1)
	}
	\tcc{No valid orientations in this case}
	\If{$\exists u, indeg\left[u\right]>deg\left[u\right]$ or $indeg\left[u\right]<0$  }{
		\Return(0)
	}
	\tcc{Examine the conditions yielding unique orientations}
	\For{$u\leq n$}
	{   
		\If{$indeg[u]=0$ \tcp*{$u$ admits only new outdirerected edge orientations}} 
		{
			\For{all edges $(u,v) \in E$}{
				$indeg[v]=indeg[v]-1$\\
				$E'$=$E\backslash \{(u,v)\}$\\
				$newdeg=$vertex degree of graph $G'(V,E')$\\
				\Return(orient($n$, $Enew$, indeg, $newdeg$))
			}
		}
		\ElseIf{$indeg[u]=deg[u]$ \tcp*{$u$ admits only new indirerected edge orientations}}
		{
			\For{all edges $(u,v) \in E$}{
				$indeg\left[u\right]=indeg\left[u\right]-1$\\
				$E'$=$E\backslash \{(u,v)\}$\\
				$newdeg=$vertex degrees of graph $G'(V,E')$\\
				\Return(orient($n$, $E'$, $indeg, newdeg$))
			}
		}
	}
	\tcc{No more unique orientations exists: set both orientations for 1st edge}
	$(u,v)=E\left[1\right]$ \\
	$indeg1\left[u\right]=indeg\left[u\right]-1$\\
	$indeg2\left[v\right]=indeg\left[v\right]-1$\\
	$E'$=$E\backslash \{(u,v)\}$\\
	$newdeg=$vertex degree of graph $G'(V,E')$\\
	orient1=orient($n$, $E'$, $indeg1, newdeg$)\\
	orient2=orient($n$, $E'$, $indeg2, newdeg$)\\
	\Return(orient1+orient2)
\end{algorithm}

Let us demonstrate our method examining one Laman and one Geiringer graph.
\begin{example}
	Here are two examples of this counting method in the case of $L56$ graph in dimension $2$ and $G48$ \footnote{The graphs in this example and the following ones are named after their class ($L$ for Laman and $G$ for Geiringer) and the number of their embeddings in the correspondent euclidean space as in \cite{belt}.} in dimension $3$, which are both 7-vertex graphs (see Figure~\ref{fig:orient}).
	Graph $L56$ has $56$ complex embeddings in the plane and $64$ embeddings on the sphere, while $G48$ has $48$ embeddings in $\CC^3$ (these numbers coincide also with the maximum number of real embeddings) \cite{belt,GKT17,count_sphere}.
	The mixed volumes of the algebraic systems are $64$ and $48$ respectively.
	
	The dashed lines indicate the fixed edges of $K_d$.
	The edge direction for any edge that includes the fixed points is always oriented towards the non-fixed vertex.
	This yields a \emph{unique orientation} up to the fixed $K_d$, which is coloured in red for both graphs.
	The rest of the graph admits $|\mathcal{H}_{K_2}|=2$ orientations for $L56$, while the number of different orientations for $G48$ is $|\mathcal{H}_{K_3}|=3$.
	So the m-B{\'e}zout bound is $2^{7-2} \cdot 2=64$ for $L56$ and $2^{7-3} \cdot 3=48$ for $G48$.
\end{example}

We have implemented a software tool in Python to count the number of orientations for an arbitrary graph given the desired indegrees.
The basic part of this code (see Algorithm~\ref{alg:orient}) is to decide recursively which choices of direction are allowed in every step \cite{code_mBezout}.

\subsection{Computing m-B{\'e}zout bounds using the permanent}

The permanent of an $m \times m$ matrix $A=(a_{i,j})$ is defined as follows:
\begin{equation}\label{eq:permanent}
\per(A)= \sum\limits_{\sigma \in S_m} \prod_{i=1}^{m} a_{i,\sigma(i)},
\end{equation}
where $S_m$ denotes the group of all permutations of $m$ integers.

One of the most efficient ways to compute the permanent is by using Ryser's formula~\cite{ryser}:
\begin{equation}\label{eq:Ryser}
\per(A)= \sum \limits_{M \subseteq \{1,2,\dots ,m\}}  (-1)^{m-|M|} \prod_{i=1}^{m} \sum\limits_{j \in M} a_{i,j} .
\end{equation}
There is a very relevant relation between $\text{per}(A)$ and the m-B{\'e}zout bound, see \cite{emvid}:

\begin{theorem}\label{thm:perm}
	Given a system of algebraic equations and a partition of the variables in $k$ subsets, as in Theorem~\ref{thm:mBezout}, we define the square matrix $A$ with $m=\sum\limits_{j=1}^{k} m_j$ rows, where each set of variables corresponds to a block of $m_j$ rows.
	Let $\alpha_{i,j}$ be the degree of the $i$-th equation in the $j$-th set of variables.
	The columns of $A$ correspond to the equations, where the subvector of the $i$-th column associated to the $j$-th set of variables has $m_j$ entries, all equal to $\alpha_{i,j}$.
	Then, the m-B{\'e}zout bound of the given system is equal to 
	\begin{equation}\label{per}
	\frac{1}{m_1! \, m_2 ! \cdots m_k!} \cdot \mbox{\rm per}(A) .
	\end{equation} 
\end{theorem}

We will refer to $A$ matrix as the \textit{m-B\'ezout matrix} of a polynomial system.
This implies that in the case of minimally rigid graphs, we obtain a square m-B\'ezout matrix $A$ with columns associated to the equations of non-fixed edges, and $n-d$ blocks of $d$ rows each, corresponding to the non-fixed vertices.
An entry $(r,c)$ is one if the vertex corresponding to $r$ is adjacent to the edge corresponding to the equation indexing $c$, otherwise it is zero.  
This is an instance of a $(0,1)$-permanent.
Therefore Theorem~\ref{thm:perm} gives the coefficient 
$$
mB_{E}(G,K_d)=\left(\frac{1}{d!}\right)^{n-d} \cdot \mbox{\rm per}(A),
$$
in bounding the system's roots, since all $m_i=d$, while $k=n-d$.
The effect of the magnitude equations implies that we should multiply $mB_{E}(G,K_d)$ by $2^{n-d}$ as in Subsection~\ref{sec:SphmBez}.
This yields the corollary below.

\begin{corollary}
	The m-B{\'e}zout bound of the sphere equations for an $n$-vertex rigid graph in $d$ dimensions up to a given $K_d$ is exactly 
	\begin{equation} \label{mbe}
	mB(G,K_d)= \displaystyle \left(\frac{2}{d!}\right)^{n-d} \cdot \mbox{\rm per}(A),
	\end{equation} 
	assuming that matrix $A$ is the m-B\'ezout matrix up to $K_d$ defined as above.
\end{corollary}

The permanent formulation for the computation of the m-B\'ezout bound gives us another way to prove Corollary~\ref{cor:orientH1}.

\begin{corollary}
	Let $G$ be a minimally rigid graph in $\CC^d$ with $n$ vertices and $A$ be  its $(m \times m)$ m-B\'ezout matrix up to a fixed $K_d$.
	Then, for every  graph $G^*$ obtained by an H1 operation on $G$, the permanent of its m-B\'ezout matrix $A^*$ up to the same $K_d$ is $\per(A^*)=d! \cdot \per(A)$.
\end{corollary}
\begin{proof}
	Without loss of generality, we consider that the last $d$ rows of matrix $A^*$ represent the new vertex, while the last $d$ columns of this matrix represent the edges adjacent to this vertex, since matrix permanent is invariant under row or column permutations.
	The rest of the matrix is the same as $A$.
	This yields the following structure:
	$$
	A^*= \begin{pmatrix}
	A & A'   \\\
	\mathbf{0} & \mathbf{1}
	\end{pmatrix} 
	$$
	where $\mathbf{0}$ is a $(d \times m)$ zero submatrix, $\mathbf{1}$ is a $(d \times d)$ submatrix with ones and $A'$ the $(m \times d)$ submatrix of the new edge columns without the new rows.
	It is clear from the definition of the permanent (See Equation~\ref{eq:permanent}), that  column permutations that do not include a zero entry are counted as 1 in this sum, while if they include a zero entry the product is zero.
	The only column permutations that do not include a zero entry for the $d$ last rows are those that are related to the $d$ last edges, so there are $d!$ nonzero column permutations for this block of rows.
	This means that the permutations for the other $m$ rows exclude the last three columns, so they are exactly $\per(A)$ permutations in this case.
	Thus, $\per(A^*)=d! \cdot \per(A)$.
\end{proof}

Since $mB(G^*,K_d)=\displaystyle\left(\frac{2}{d!}\right)^{n+1-d} \cdot \per(A^*)$, it follows that 
$$
mB(G^*,K_d)=2\cdot mB(G,K_d),
$$
as in Corollary~\ref{cor:orientH1}. \\

Let us give an example of this counting method for a minimally rigid graph.
\begin{figure}[htp!]
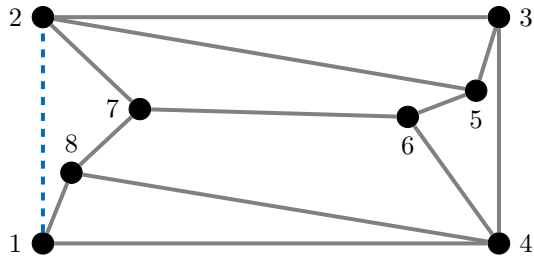

	\centering \Lmaxeight
	\caption{The $L136$ graph. The dashed edge is the fixed one. \label{fig:L136}}
\end{figure}
\begin{example}
	We use the $L136$ graph to provide an example for this formulation (other examples can be found in \cite{code_mBezout}).
	$L136$ is the 8-vertex Laman graph with the maximal embedding number $c_2(G)=136$ among all Laman graphs with the same number of vertices \cite{Joseph_lam}. 
	On $S^2$, it has $192$ complex embeddings, which is also maximum (but not unique), since there is another graph sharing the same $c_{S^2}(G) $.
	
	The m-B\'ezout matrix $A_{L136}$ for this graph for the fixed edge $(1,2)$ is the following:
	\begin{center}	\footnotesize
		\hspace*{-9mm} \begin{tabular}{c|cccccccccccc|}
			& $(1,4)$ & $(1,8)$ & $(2,3)$ & $(2,5)$ & $(2,7)$ & $(3,4)$ & $(3,5)$ & $(4,6)$ & $(4,8)$ & $(5,6)$ & $(6,7)$ & $(7,8)$
			\\ \hline
			$x_3$ & 0 & 0 & 1 & 0 & 0 & 1 & 1 & 0 & 0 & 0 & 0 & 0 \\
			$y_3$ & 0 & 0 & 1 & 0 & 0 & 1 & 1 & 0 & 0 & 0 & 0 & 0 \\ 
			$x_4$ & 1 & 0 & 0 & 0 & 0 & 1 & 0 & 1 & 1 & 0 & 0 & 0 \\
			$y_4$ & 1 & 0 & 0 & 0 & 0 & 1 & 0 & 1 & 1 & 0 & 0 & 0 \\
			$x_5$ & 0 & 0 & 0 & 1 & 0 & 0 & 1 & 0 & 0 & 1 & 0 & 0 \\
			$y_5$ & 0 & 0 & 0 & 1 & 0 & 0 & 1 & 0 & 0 & 1 & 0 & 0 \\
			$x_6$ & 0 & 0 & 0 & 0 & 0 & 0 & 0 & 1 & 0 & 1 & 1 & 0 \\
			$y_6$ & 0 & 0 & 0 & 0 & 0 & 0 & 0 & 1 & 0 & 1 & 1 & 0 \\
			$x_7$ & 0 & 0 & 0 & 0 & 1 & 0 & 0 & 0 & 0 & 0 & 1 & 1 \\
			$y_7$ & 0 & 0 & 0 & 0 & 1 & 0 & 0 & 0 & 0 & 0 & 1 & 1 \\
			$x_8$ & 0 & 1 & 0 & 0 & 0 & 0 & 0 & 0 & 1 & 0 & 0 & 1\\
			$y_8$ & 0 & 1 & 0 & 0 & 0 & 0 & 0 & 0 & 1 & 0 & 0 & 1\\
			\hline
		\end{tabular}
	\end{center}
	and its permanent is $\per(A_{L136})=192$, which gives the m-B\'ezout bound since $d=2$.
\end{example}

\subsection{Improving asymptotic upper bounds}
We make use of both approaches to improve upon asymptotic upper bounds on the embedding number.
Despite various algebraic formulations and approaches on root counts, it is surprising that the best existing asymptotic upper bounds on the embedding number of minimally rigid graphs are in the order of $O\left(2^{dn} \right)$ and, thus, essentially equal the most straightforward bound on the most immediate system, namely B\'ezout's bound on the edge equation system \cite{Borcea,Steffens}.

On the other hand, the lower bounds keep on improving (please refer to the Introduction) but still an important gap remains.
This is indicative of the hardness of the problem.

First, we make use of the following proposition on the asymptotic bounds for the orientations of planar graphs in order to improve the asymptotic upper bound of planar Geiringer graphs, which are the only fully characterized class of minimally rigid graphs in 3d space, and hence of special interest.

\begin{proposition}[Felsner and Zickfeld~\cite{Felsner}]
	The number of indegree constrained orientations of a planar graph is bounded from above by 
	\begin{equation}
	2^{n-4} \cdot \prod_{u \in I} \left(2^{-deg(u)+1} \cdot \binom{deg(u)}{indeg(u)} \right)
	\end{equation}
	where $I$ is an independent set of the graph, $deg(u)$ and $indeg(u)$ are respectively the degree and the indegree of a vertex $u$.
	Furthermore, in the case of $indeg(u)=3$ this bound asymptoticaly behaves as $3.5565^n$.
\end{proposition}

Given the relation between m-B\'{e}zout bounds and graph orientations (see Section~\ref{sec:comb}), this proposition leads to the following improvement upon the asymptotic upper bound for the number of embeddings of the subclass of planar Geiringer graphs.

\begin{theorem}
	Planar Geiringer graphs have at most $O\left(7.1131^n\right)$ embeddings.
\end{theorem}

We also employ the permanent to obtain asymptotic improvement upon B\'ezout's asymptotic bound for $d\ge 5$ by using the following  bound.

\begin{proposition}[Br\`egman~\cite{Bre73}, Minc~\cite{Minc63}]
	For a $(0,1)$-permanent $A$ of dimension $m$, it holds:
	\begin{equation}\label{per_bound}
	\per(A) \leq \prod_{j=1}^m \left(r_j ! \right)^{1/r_j} ,
	\end{equation}
	where $r_j$ is the sum of the entries in the $j$-th column (or the $j$-th row).
\end{proposition}

This leads to the following result.

\begin{theorem}
	For $d\geq 5$ the B\'ezout bound is strictly larger than the m-B\'ezout bound given by Equation~(\ref{mbe}) for any fixed $K_d$.
	Given a fixed $d$, the new asymptotic upper bound  derived from the Br\`egman-Minc inequality is
	$$
	O\left( \left(2\cdot \displaystyle \frac{\sqrt{(2d) !}}{d !}   \right)^n \right) .
	$$
	
\end{theorem}
\begin{proof}
	In this proof $Be_d(n)$ and $mB_d(n)$ denote the B\'ezout and the maximal m-B\'ezout bound of minimally rigid graphs in $\CC^d$ with $n$ vertices respectively.
	Since the number of edge equations for minimally rigid graphs with $n$ vertices is $n\cdot d- d^2$, the B\'ezout bound is 
	$$ 
	Be_d(n) = 2^{n\cdot d- d^2}.
	$$
	The sum of columns for the permanent that computes the m-B\'ezout bound is $r_j = d$ for the edges that include one fixed vertex and one non-fixed vertex and $r_j = 2d$ for these that include two non-fixed vertices.
	We denote these sets of edges $E_{f.}$ and $E_{n.f.}$ respectively.
	Applying the Br\`egman-Minc bound and Equation~(\ref{mbe}) 
	we get 
	\begin{equation}	\label{mbper}
	mB_d(n)\leq \left(\frac{2}{d!}\right)^{n-d} \cdot \prod_{i=1}^{E_{f.}} (d!)^{1/d} \prod^{E_{n.f.}} (2d!)^{1/2d}\leq  \left( 2\cdot \displaystyle \frac{\sqrt{(2d) !}}{d !}   \right)^{n-d} .
	\end{equation}
	Combining these bounds we get a sufficient condition for $Be_d(n) > mB_d(n)$:
	\begin{align*}
	2^{n\cdot d- d^2} > \left( 2\cdot \displaystyle \frac{\sqrt{(2d) !}}{d !} \right)^{n-d} 
	\Leftrightarrow\, 2^{2d-2} \cdot (d !)^2 > (2d) ! 
	\end{align*}
	Robbins' bound on Stirling's approximation \cite{Robbins} yields the following:
	$$ 
	\sqrt{2\pi} \cdot d^{d+1/2} \cdot e^{-d} \cdot e^{R_-}< d! < \sqrt{2\pi} \cdot d^{d+1/2} \cdot e^{-d} \cdot e^{R_+},
	$$
	where $R_+= \displaystyle \frac{1}{12d}$ and $R_-= \displaystyle \frac{1}{12d+1} $. We now derive the following inequalities:
	\begin{align*}
	2^{2d-2} \cdot (d !)^2 \quad & > \; 2^{2d-2} \cdot 2\pi \cdot d^{2d+1} \cdot e^{-2d} \cdot e^{2R_-} & \\
	& > \sqrt{2\pi} \cdot 2^{2d+1/2} \cdot d^{2d+1/2} \cdot e^{-2d} \cdot e^{R_+/2} & > \; 2d ! 
	\end{align*}
	that lead to a sufficient condition for $Be_d(n) >mB_d(n)$ to hold:
	\begin{equation}	\label{cond}
	\sqrt{d} > \frac{4}{\sqrt{\pi}} \cdot e^{R_+/2-2R_-} ,
	\end{equation}
	which is true for every integer $d\geq 5$.
	
	Additionally, inequality (\ref{mbper}) leads directly to the asymptotic bound 
	$$
	mB_d(n) \in O\left( \left(2\cdot \displaystyle \frac{\sqrt{(2d)!}}{d !}   \right)^n \right) 
	$$
	for any given $d$.
	
\end{proof}

\begin{table}[htp]
	\caption{Comparing B\'ezout's and permanent's asymptotic bounds (last two rows) for increasing $d$ (top row),}
	\begin{tabular}{|c|c|c||c|c|c|c|c|c||c|}
		\hline
		2 & 3 & 4 & 5 & 6 & 7 & 8 & 9 & 10 &30\\
		\hline\hline
		$4^n$ & $8^n$ & $16^n$ & $32^n$ & $64^n$ & $128^n$ & $256^n$ & $512^n$ & $1024^n$ & $(1.07\cdot 10^9)^{n}$ \\
		\hline
		$4.9^n$ & $8.9^n$ & $16.7^n$ & $31.7^n$ & $60.8^n$ & $117.2^n$ & $226.9^n$ & $441^n$ & $860^n$ & $(6.88\cdot 10^8)^n$\\
		\hline
	\end{tabular}
\end{table}

\subsection{Runtimes}

The computation of the m-B{\'e}zout bounds using our combinatorial algorithm up to a fixed $K_d$ is much faster than the computation of mixed volume and complex embeddings.
In order to compute the mixed volume we used \texttt{phcpy} in \texttt{SageMath} \cite{phcpy} and we computed  complex solutions of the sphere equations using \texttt{phcpy} and and \texttt{MonodromySolver} \cite{monodromySolver}.
Let us notice that, \texttt{MonodromySolver} seems to be faster than mixed volume software we used in the case of Geiringer graphs.
We also compared our runtimes with the combinatorial algorithm that counts the exact number of complex embeddings in $\CC^2$ \cite{Joseph_lam}.

\begin{table*}[htp!] \hspace*{-6mm} 
	\caption{Runtimes of different algorithms on graphs with maximal $c_d(G)$ up to $n=11$ and up to $n=10$ for Laman and Geiringer graphs, resp.
		We compute $c_2(G)$ by \cite{Joseph_lam} and $c_3(G)$ by \texttt{phcpy} \cite{sourceCode} (fails to find all solutions for $n> 11$).
		Also runtimes for computing $c_2(G)$, $c_3(G)$ by \texttt{MonodromySolver}.
		We compute MV by \texttt{phcpy}, m-B\'ezout by \texttt{Maple}'s permanent and our \texttt{Python} code \cite{code_mBezout}.
		Computation of the m-Bézout and MV is up to a fixed $K_d$ (edges or triangles).\vspace*{2mm}}
	\centering
	\begin{tabular}{|c||c|c|c|c|c|}
		\hline
		& \multicolumn{5}{c|}{\bf Laman graphs} \\
		\hline
		$n$ 	&  \parbox[t]{1cm}{combin.\\ $c_2(G)$} & \parbox[t]{1.45cm}{\texttt{Monodromy} \\ \texttt{Solver}} & \parbox[t]{8mm}{\texttt{phcpy}\\ MV} & \parbox[t]{7mm}{\texttt{Maple}'s \\ perm.} & \parbox[t]{14mm}{m-B{\'e}zout \\ \texttt{Python}}\\ 
		\hline\hline
		
		6 & 0.0096s & 0.2334s & 0.0024s & 0.0003s & 0.0009s  \\
		\hline 
		7 & 0.0153s & 0.566s & 0.006s & 0.00045s & 0.0012s  \\
		\hline
		8 & 0.0276s & 1.373s & 0.0122s & 0.00065 & 0.002s  \\
		\hline 
		9 & 0.066s& 4.934s & 0.0217s & 0.0018s & 0.0032s  \\
		\hline
		10 & 0.176s & 12.78s & 0.043s & 0.0053s & 0.0045s  \\
		\hline 
		11 & 0.558s & 46.523s & 0.17s & 0.0077s & 0.0074s  \\
		\hline 
		12 & 6.36s & 2m47s & 0.39s & 0.049s & 0.013s \\
		\hline 
		18 & 17h 5m & - & 1h 34m & 24s & 0.115s \\
		\hline
	\end{tabular}
	
	\vspace*{0.3 cm}
	\begin{tabular}{|c||c|c|c|c|c|}  
		\hline
		&  \multicolumn{5}{c|}{ \bf Geiringer graphs} \\
		\hline
		$n$ & \parbox[t]{8mm}{\texttt{phcpy}\\ solver} & 
		\parbox[t]{1.45cm}{\texttt{Monodromy} \\ \texttt{Solver}} &
		\parbox[t]{8mm}{\texttt{phcpy}\\ MV} & \parbox[t]{7mm}{\texttt{Maple}'s\\ perm.} & \parbox[t]{14mm}{m~-~B{\'e}zout \\ \texttt{Python}}   \\
		\hline\hline
		
		6 & 0.652s & 0.141s & 0.00945s & 0.0002s  & 0.0097s  \\
		\hline 
		7 &  3.01s & 0.584s & 0.041s & 0.001s & 0.00165s  \\
		\hline
		8 &  20.1s & 2.297s & 0.425s & 0.0025s & 0.00266s  \\
		\hline 
		9  & 2m 33s & 14.97s & 3.42s & 0.0075s & 0.006s \\
		\hline
		10  & 16m 1s & 1m23s & 1m 12s & 0.08s & 0.0105s \\
		\hline 
		11  & 2h 14m & 9m22s & 27m31s   & 0.49s & 0.024s \\
		\hline 
		12  & - & 1h22m  & $>6$days & 0.96s & 0.06s\\
		\hline 
		
	\end{tabular} 
	
\end{table*}

We will try to give some indicative cases for which we compared the runtimes.
For example, computing the mixed volume of the spherical embeddings up to one fixed edge for the maximal 12-vertex Laman graph for a given fixed $K_2$ takes around 390ms, while our algorithm for the m-B\'{e}zout bound required 13ms.
If we wanted to compute mixed volumes up to all fixed $K_2$ we needed 8.6s, while the m-B{\'e}zout computation took 270ms.
The runtime for the combinatorial algorithm that computes the number of complex embeddings is 6.363s for the same graph.

For larger graphs i.e.\ 18-vertex graphs, the combinatorial algorithm  may take  $\sim 17$h to compute the number of complex embeddings in $\CC^2$.
We tested a 18-vertex graph that did not require more than 0.12s to compute one m-B{\'e}zout bound and 4s to compute  m-B{\'e}zout bounds up to all choices of fixed edges.

In dimension 3 our model was the Icosahedron graph, which has~12 vertices.
The computation of the mixed volume took more than 6~days in this case, while our algorithm needed 60ms to give exactly the same result (54,272).
\texttt{MonodromySolver} could track all $ 54,272$ solutions in $\sim 1.3$ hour, while Gr\"obner basis computations failed multiple times to terminate.

Computing the permanent required more time compared to our algorithm.
For the Icosahedron the fastest computation could be done using \texttt{Maple}'s implementation in \texttt{LinearAlgebra} package.
It took $\sim 0.96$s to compute the permanent up to a given fixed triangle with this one.
On the other hand the implementations in \texttt{Python} and \texttt{Sage} took much more time for the same graph ($\sim 8$m and $\sim10$m respectively).

This seems reasonable since the combinatorial algorithm has to check at most $2^{m}$ cases, while according to \cite{emvid} the complexity to compute the  permanent using Ryser's formula is in the order of $m^2 \cdot 2^{m}$.

\section{On the exactness of m-B\'ezout bounds}\label{sec:exact}

In this section we examine the exactness of m-B\'ezout bounds.
First, we use already published results (in the cases of $\CC^2$ and $\CC^3$) \cite{GKT17,belt} and our own computations (in $S^2$) to compare m-B\'ezout bounds with the mixed volumes and the actual number of embeddings.
Then, we present a general method to decide if the m-B\'ezout bound of a minimally rigid graph is tight or not based on Bernstein's second theorem on mixed volumes \cite{Bernshtein1975} without directly computing the embeddings.
We consider that the latter may be a first step to establish the existence of a particular class of graphs with tight m-B\'ezout bounds.

\subsection{Experimental results}\label{sec:res}

We compared m-B{\'e}zout bounds with the number of embeddings and mixed volumes using existing results \cite{GKT17,etv,belt} for the embeddings in $\CC^2$ and $\CC^3$.
We also computed the complex solutions of the equations that count embeddings on $S^2$ for all Laman graphs up to 8 vertices and a selection of graphs with up to 12 vertices that have a large number of embeddings.
We remind that in general the m-B\'ezout bound is not unique up to all choices of a fixed $K_d$.
It is natural to consider the minimal m-B\'ezout bound as the optimal upper bound of the embeddings for a given graph.
Let us notice that we checked if the m-B\'ezout is minimized when the fixed $K_d$ has a maximal sum of vertex degrees or when the vertex with the maximum degree belongs to the fixed $K_d$.
There are counter-examples for both of these hypotheseis.

\paragraph{Mixed volume and m-B{\'e}zout bound.} 
In all cases we checked in $\CC^3$ and $S^2$, the m-B{\'e}zout bound up to a fixed $K_d$ is exactly the same as the mixed volume up to the same fixed points.
There are some cases in $\CC^2$ for which the m-B{\'e}zout bound is bigger than the mixed volume for certain choices of $K_2$.
We shall notice that these cases do not correspond to the minimal m-B{\'e}zout bound for the given graph (thus the minimum m-B{\'e}zout and the minimum mixed volume are the same for these graphs).

\paragraph{Spatial embeddings and the m-B{\'e}zout bound.}

As shown in~\cite{belt}, there are many cases for which the bounds of the sphere equations are larger than the actual number of complex embeddings.
Nevertheless, we observed that for all planar graphs up to $n=11$ the number of complex embeddings is exactly the same as the mixed volume bound and therefore the m-B{\'e}zout bound, while in the non-planar case the bounds are generally not tight.
What is also interesting is that the m-B{\'e}zout bound is invariant for all choices of fixed triangles in the case of planar Geiringer graphs.

\begin{table*}[!htb] \begin{center}	
		\caption{Mixed volumes, complex embedding numbers, and m-B{\'e}zout bounds for embeddings of Laman graphs in $\CC^2$ and $S^2$.
			These graphs have the maximal number of embeddings in $\CC^2$.
			The 12-vertex maximal Laman graph is the first non-planar in this category. \vspace*{1.5mm}}
		\begin{tabular}{|c||c|c||c|c||c|}			\hline
			$\bm{n}$ & $\bm{\text{\textbf{MV}}_{2d}}$ & $\bm{c_2(G)}$ & $\bm{\text{\textbf{MV}}_{S^2}}$ & $\bm{c_{S^2}(G)}$ & \textbf{mB{\'e}zout} \\
			\hline\hline
			6 & 32 & 24 & 32 & 32 & 32 \\
			\hline
			7 & 64 & 56 & 64 & 64 & 64 \\
			\hline
			8 & 192 & 136 & 192 & 192& 192 \\
			\hline
			9 & 512 & 344 & 512 & 512 & 512 \\
			\hline
			10 & 1536 & 880 & 1536 & 1536 & 1536 \\
			\hline
			11 & 4096 & 2288 & 4096 & 4096 & 4096 \\
			\hline
			12 & 15630 & 6180 & 15630 & \Red{\em 8704} & 15630 \\
			\hline 
\end{tabular} \end{center} \end{table*}

\paragraph{Embeddings of Laman graphs and the m-B{\'e}zout bound.} 
For Laman graphs, the m-B{\'e}zout bound diverges from the number of actual embeddings in $\CC^2$ more than in the case of Geiringer graphs.
That happens both for planar and non-planar graphs.
On the other hand the number of spherical embeddings of planar Laman graphs coincides with the minimum m-B{\'e}zout bound for a vast majority of cases (all planar graphs up to 6 vertices,  64/65 7-vertex planar graphs and 496/509 8-vertex planar graphs).
Notice that the m-B{\'e}zout bounds for different choices of the fixed edge are, in general, different for planar Laman graphs.

\subsection{Using Bernstein's second theorem}

Our computations indicate that the m-B{\'e}zout bound is tight for almost all planar Laman graphs in $S^2$ and all planar Geiringer graphs.
Therefore, we decided to apply Bernstein's second theorem to establish a method that determines whether this bound is exact.
We believe that a generalization of this method may show whether the experimental results provably hold for certain classes of graphs.
In this subsection for reasons of simplicity, in all examples we will use the variables $x_i,y_i,s_i$ and $x_i,y_i,z_i$ for $\CC^2$ and $S^2$ respectively (instead of $x_{i,1}, \cdots x_{i,d},s_i$, used earlier).

Bernstein has given discriminant conditions such that the BKK bound (Theorem~\ref{thm:BKK}) for a given algebraic system may be tight.
In order to state Bernstein's conditions we first need the following definition.

\begin{definition}[Initial form]
	Let $P$ be a polytope in $\RR^m$, $w$ be a vector in $\RR^m$ and $f = \sum\limits_{\alpha \in \mathcal{A}} c_{\alpha} {\bf x}^{\alpha}$ be a polynomial in $\CC[x_1,x_2, \dots ,  x_m ]$.
	Let also $P^{w}= \{v \in P\, | \,\langle v,w \rangle \leq \langle u,w \rangle, \, \forall u \in P \} $ be the subset of $P$ that minimizes the inner product with $w$.
	The initial form of $f$ with respect to $w$ is defined as 
	$$
	f^w= \sum \limits_{\beta \in P^w \cap \mathcal{A}} c_{\beta} \mathbf{x}^{\beta}. 
	$$
\end{definition}

The initial form $f^w$ contains precisely the monomials whose exponent vector minimizes the inner product with $w$ and excluding the others.
Clearly $P^w$ is a face of $P$ and $w$ is an inner normal to face $P^w$.
Hence, the algebraic system comprised of initial forms for a face normal $w$ shall be called \textit{face system}.

The necessary and sufficient condition of BKK exactness is stated below. Recall that the Minkowski sum of sets is the set of all sums between elements of the first and elements of the second set. The Minkowski sum of convex polytopes is a convex polytope containing the vector sums of all points in the two summand polytopes.

\begin{theorem}[Bernstein's second theorem \cite{Bernshtein1975}]\label{thm:Ber2}
	Let $ (f_i)_{1 \leq i \leq m }$ be  a system of equations in $\CC[x_1,x_2, \dots , x_m ]$ and let
	$$
	P=\sum_{i=1}^m NP(f_i)
	$$
	be the Minkowski sum of their Newton Polytopes.
	The number of solutions of $ (f_i)_{1 \leq i \leq m } $ in $(\CC^*)^m$
	equals exactly its mixed volume (counted with multiplicities) {if and only if}, for all $w \in \RR^m$, such that $w$ is a face normal of $P$, the system of equations $(f^w_i)_{1 \leq i \leq m }$ has no solutions in $(\CC^*)^m$.	
\end{theorem}

Let us note that although there is an infinite number of vectors that may appear as inner normals, Bernstein's condition can be verified  choosing only one inner normal vector for every different face of $P$.

The results in Section~\ref{sec:res} motivated us to examine these conditions closely in order to determine when the m-B{\'e}zout bound is exact.
The first step is to use Newton polytopes whose mixed volume equals to the m-B{\'e}zout  bound (see for example \cite{SomWam})  since they are simpler than the Newton polytopes of the sphere equations.

\begin{definition}
	We set $e_i=(0,0,\dots, \underset{\text{i-th  position}}{1}, \dots,0 )$ and let $T_k^u$ be the simplex defined as the convex hull of the set 
	$$
	\{\mathbf{0}, e_{k\cdot (u-1)+1},e_{k\cdot (u-1)+2}, \dots, e_{k\cdot (u-1)+k} \} ,
	$$
	where $\mathbf{0}=(0,0,\dots,0)$ is the origin.
	Let $X_u$ be the $u$-th set of variables under a partition of all variables, with set cardinality $d_u=|X_u|$.
	Then $T_{d_u}^u$ is the simplex that corresponds to the variables of this set.
	The \emph{m-B{\'e}zout Polytope} of a polynomial, with respect to a partition of the variables, is the Minkowski sum of the $T_{d_u}^u$ for all $u$, such that each simplex is scaled by the degree of the polynomial in $X_u$.
\end{definition}

For a multihomogeneous system, simplices $T^u_{d_u}$ belong to complementary subspaces.
Then, each m-B{\'e}zout Polytope is the Newton polytope of the respective equation.
For general systems, our procedure amounts to finding the smallest polytopes that contain the system's Newton polytopes and can be written as Minkowski sum of simplices lying in the complementary subspaces specified by the variable partition.

In the case of rigid graphs in $\CC^d$, every set of variables has $d+1$ elements. 
Thus, the m-B{\'e}zout Polytope of the magnitude equations for a vertex $u$ is $2 \cdot T_{d+1}^u$, while the m-B{\'e}zout Polytope of the equation for edge $(u,v)$ is $T_{d+1}^u+T_{d+1}^v$.
This implies that the Minkowski sum of the m-B{\'e}zout Polytopes for the sphere equations of a minimally rigid graph $G=(V,E)$ is exactly
$$
B_G= \sum\limits_{u \in V'} ( \deg(u) +2) \cdot T_{d+1}^u ,
$$
where $\deg(u)$ is the degree of vertex $u$ in the graph and $V'$ the set of non-fixed edges.

In general, it is hard to compute the Minkowski sum of polytopes in high dimension.
But in the case of the m-B{\'e}zout Polytopes the following theorem describes the facet normals of $B_G$.
\begin{theorem}
	Let $G=(V,E)$ be a minimally rigid graph in $\CC^d$ and $B_G$ be the Minkowski sum defined above.
	The set of the inner normal vectors of the facets of $B_G$ are exactly
	\begin{itemize}
		\item all unit vectors $e_i$, and
		\item the $n-d$ vectors of the form 
		$$
		\delta_u=  \sum\limits_{j=1}^{d+1} -e_{(d+1)\cdot(u-1)+j} =(0,0,\dots, -1,-1,\dots,-1,\dots, 0 ),
		$$
		where there are $d+1$ nonzero entries corresponding to the variables that belong to the $u$-th variable set.
	\end{itemize}
\end{theorem}

\begin{proof}
	Since each $T_{d+1}^u$ belongs to a complementary subspace, $B_G$ can be seen as the product of polytopes $\prod\limits_{u \in V'} ( \deg(u)+2) \cdot \Delta_{d+1}$, where $\Delta_{d+1}$ is the unit $(d+1)$-simplex \footnote{The idea of using the product of polytopes is derived by a proof for the mixed volumes corresponding to the weighted m-B\'ezout bound in \cite{Pinaki}}.
	The inner normal vectors of the facets of $\Delta_{d+1}$ in $\RR^{d+1}$ are the unit vectors $e_i$ and $\delta=\sum\limits_{j=1}^{d+1} -e_{j}$ in $\RR^{d+1}$. 
	The theorem follows since the normal fan of a product of polytopes is the direct product of the normal fans of each polytope~\cite{Ziegler}.
\end{proof}

This theorem yields a method to find the H-representation of $B_G$, in other words the polytope is described as the intersection of linear halfspaces and the respective equations are given by the theorem.
In all cases where $MV=mBe$, the polytopes $B_G$ can be used instead of the Newton Polytopes of the equations.

The verification of Bernstein's second theorem requires a certificate for the existence of roots of face systems for every face of $B_G$, where faces range from vertices of dimension~0 to facets of codimension 1.
We propose a method that confirms or rejects Bernstein's condition checking a much smaller number of systems based on the form of facet normals.
For this, we shall distinguish three cases below.

The normal of a lower dimensional face can be expressed as the vector sum of facet normals, whose cardinality actually equals the face codimension.
This means that we need to verify normals distinguished in the following three cases:
\begin{enumerate}
	\item vector sums of one or more ``coordinate" normals $e_i$'s,
	\item vector sums of one or more ``non-coordinate" normals $\delta_u$'s,
	\item ``mixed" vector sums containing both $e_i$'s and $\delta_u$'s.
\end{enumerate}

Notice that since there are $(d+2)\cdot (n-d)$ different normals, in order to check all resulting face systems, $2^{(d+2)\cdot (n-d)}$ computations are required.
We now examine each of these three cases separately, in order to exclude a very significant fraction of these computations.

\paragraph{First case (coordinate normals).}
Let $F=(f_i )_{1\leq i \leq m}$ be the system of the sphere equations, let the initial forms be $f_i^{e}$ for some normal $e$, and let $F^e$ be the resulting face system.
We will deal with the coordinate normals case starting with an example.

\begin{example}\label{ex:e_i}
	We present the equations of face system $F^{e_1}$ in $\CC^2$.
	Normal $e_1$ corresponds to variable $x_1$.
	This means that the inner products with the exponent vectors of the monomials in the magnitude equation $f_1 = x_1^2+y_1^2-s_1$ are $2,0,0$.
	Thus, $f_1^{e_1} = y_1^2-s_1$, excluding the monomial $x_1^2$.
	In the case of the edge equation $f_{(1,2)}= s_1+s_2-2 (x_1x_2+y_1y_2)+\lambda_{1,2}^2$, for a generic edge length $\lambda_{1,2}$, the inner products are $0,0,1,0,0$.
	It follows that  $f_{(1,2)}^{e_1}= s_1+s_2-2 y_1y_2+\lambda_{1,2}^2$.
	If the degree of $x_1$ in an equation $f_i$ is zero, then $f_i^{e_1}=f_i$, since the inner product of all the exponent vectors with $e_i$ is zero.
\end{example}

This example shows that since all  $x_1$ monomials are removed, $F^{e_1}$ is an over-constrained system that has the same number of equations as $F$, but a smaller number of variables.
The same holds obviously for every $F^{e_i}$, while for $e=\sum\limits_{i \in I} e_i$ (where $I$ is an index set) the initial forms in $F^{e}$ are obtained after removing all monomials that include one or more of the variables corresponding to the $e_i$'s of the sum.
In other words, the initial forms in system $F^{e}$ can be obtained by evaluating to zero all the variables indexed by the set $I$.

\begin{lemma}
	
	Let $e$ be a sum of $e_i$ normals as described above.
	Now, $F$ does not verify Bernstein's condition in the coordinate normals case (and has an inexact BKK bound) due to system $F^e$ having a toric root $r'$, only if $F$ has a root $r$ with zero coordinate for at least one of the variables in $I$, such that the projection of $r$ to the coordinates $j\not\in I$ equals $r'$.
	
\end{lemma}

We can now exclude the case of sums of coordinate
normals from our examination, since it shall not generically occur, because the next lemma shows that $r$ has no zero coordinate.

\begin{lemma}\label{lem:e_i}
	The set of solutions of the sphere equations for a rigid graph generically lies in $(\CC^*)^{d\cdot n}$.
\end{lemma}
\begin{proof}
	We indicate by $\text{C}(G,\bm{\lambda},\{p_1,p_2,\dots p_d\})$ the set of complex embeddings for a rigid graph $G$ up to an edge labeling $\bm{\lambda}$ and $d$ fixed points $\{p_1,p_2,\dots p_d\}$.
	This set of embeddings is finite by definition.
	If there is a zero coordinate in the solution set, there exists a vector $p' \in \CC^d$, such that no zero coordinates belong to the zero set $\text{C}^*(G,\bm{\lambda},\{p_1+p',p_2+p',\dots p_d+p'\})$. 
	Since we want to verify Bernstein's condition for a generic number of complex embeddings of $G$, we can always use the second set of embeddings.
\end{proof}

This lemma excludes a total of $2^{(d+1)\cdot (n-d)}$ cases when verifying Bernstein's second theorem for a given algebraic system.

\paragraph{Second case (non-coordinate normals).}
In the second case, the inner product of exponent vectors with $\delta_u$ is minimized for all variables $X_u$ of a vertex $u$ with maximum degree.
Let us give again an example to explain this statement.

\begin{example}\label{ex:delta}
	It is an example in $\CC^2$ for face system $F^{\delta_1}$.
	The inner products for the magnitude equation $f_1 = x_1^2+y_1^2-s_1$ and the edge equation $f_{(1,2)}= s_1+s_2-2\cdot(x_1x_2+ y_1y_2)+\lambda_{1,2}^2$, $\lambda_{1,2}$ being a generic edge length, are $-2,-2,-1$ and $-1,0,-1,-1,0$ respectively.
	So, $f_1^{\delta_1} = x_1^2+y_1^2$ and $f_{(1,2)}^{\delta_1} = s_1-2\cdot(x_1x_2+ y_1y_2)$.
	If the degree of a polynomial $f_i$ in the set of variables $X_u$ is zero, then $f_i^{\delta_u}=f_i$.	
\end{example}

The number of equations of $F^{\delta_u}$ equals the number of variables.
Following Bernstein's proof on the discriminant conditions, we introduce a new 
system by applying a suitable variable transformation from the initial variable vector $\mathbf{x}$ to a new variable vector $\mathbf{t}$ with same indexing. 
This transformation uses an $m\times m$ full rank matrix $B$ such that every monomial $\mathbf{x}^{\alpha}$ is mapped to $\mathbf{t}^{B\cdot \alpha}$ (see \cite{Bernshtein1975,Cox2} for more details).
Furthermore, $|\det B| = 1$ so that the transformation preserves the  mixed volume of $F$ \cite{Cox2}.

In our case, we construct matrix $B$ with the following properties:
\begin{eqnarray*}
	B\cdot \delta_u^T & = & e_{(d+1)\cdot (u-1)+1}, \\  
	B\cdot e_{(d+1)\cdot (u-1)+j}^T & = & e_{(d+1)\cdot (u-1)+j} ,\ \forall j \in \{ 2,\dots, d+1\} , \\
	\text{det}(B) & = & \pm 1 .
\end{eqnarray*} 
The intuition behind these choices is given in Lemma~\ref{lem:Lcase2} and its proof.
This yields the following map from variables $\mathbf{x}$ to new variables $\mathbf{t}$:
\begin{align}\label{eq:deltamap}
x_{u,1} \mapsto \frac{1}{t_{u,1}}, & \,\  x_{u,2} \mapsto \frac{t_{u,2}}{t_{u,1}},  \,\ \cdots   ,\,\ s_{u} \mapsto \frac{t_{u,d+1}}{t_{u,1}} .
\end{align}

We will refer to the set of $x_{u,1}$'s as the $\delta$-\textit{variables} of $F$, since the image of their exponent vectors is the set of $\delta_u$'s, while the exponent vectors for the other variables remain same.
This transformation maps system $F(\mathbf{x})$ to a new system $\widehat{F}(\mathbf{t})$ of Laurent polynomials in the new variables.
In the case of $\CC^2$, the sphere equations are mapped as follows:
\begin{align*}
\widehat{f}_u= \frac{1}{t_{u,1}^2}+\frac{t_{u,2}^2}{t_{u,1}^2}-\frac{t_{u,3}}{t_{u,1}}, \,\,\ & \text{(magnitude equations)} \\
\widehat{f}_{(u,v)}= \frac{t_{u,3}}{t_{u,1}}+\frac{t_{v,3}}{t_{v,1}}-2\cdot \displaystyle\left(\frac{1}{t_{u,1}t_{v,1}}+\frac{t_{u,2}t_{v,2}}{t_{u,1}t_{v,1}}\right)+ \lambda_{u,v}^2 \,\,\ & \text{(edge equations) .}
\end{align*}

The degree $\alpha(\widehat{f},t_{u,1})$ of a polynomial $\widehat{f}$ with respect to the variable $t_{u,1}$ will be either zero or negative.
Let us now multiply every polynomial in $\widehat{F}(\mathbf{t})$ with each one of the monomials $t_{u,1}^{-\alpha(\widehat{f},t_{u,1})}$.
These monomials are defined as the least common multiple of the denominators in the Laurent polynomials $\widehat{f}$,  yielding the following system $\widetilde{F}(\mathbf{t})$:
\begin{equation*}\begin{split}
\widetilde{f}_u = 1+t_{u,2}^2-t_{u,1}t_{u,3}, \,\,\  \text{(magnitude equations)} \\
\widetilde{f}_{(u,v)}= t_{v,1}t_{u,3}+t_{u,1}t_{v,3}-2\cdot \displaystyle\left(1+t_{u,2}t_{v,2}\right)+\lambda_{u,v}^2\cdot t_{u,1}t_{v,1}. \,\,\  \text{(edge equations)} 
\end{split}
\end{equation*}
This transformation yields the necessary conditions to verify if the face systems of the non-coordinate normals have solutions in $(\CC^*)^m$.
We will refer to $t_{u,1}$'s as the set of $\delta$-variables of $\widetilde{F}(\mathbf{t})$, while the rest should be the $e$-variables.
Note that the transformation gives a well-constrained system, while zero evaluations of the $\delta$-variables shall result to an over-constrained system, that should have no solutions if the bound is exact.

\begin{lemma}\label{lem:Lcase2}
	There exists a sum $\delta$ of different $\delta_u$ normals, such that face system $F^{\delta}$ has a solution in $(\CC^*)^m$ if  the algebraic system $\widetilde{F}(\mathbf{t})$, which is defined above, has a zero solution for $t_{u,1}$ for one or more vertices $u$.
\end{lemma}
\begin{proof}
	Matrix $B$ is constructed to change the variables $x_{u,1}$ to variables $t_{u,1}$, for vertices $u$.
	From the definitions of $\widehat{F}(\mathbf{t})$ and $\widetilde{F}(\mathbf{t})$, it follows that given a monomial $\mathbf{t}^{\bm{\beta}}$ in a polynomial $\widehat{f}$ of $\widehat{F}(\mathbf{t})$, the inner product  $\langle \bm{\beta},B\cdot \delta_u^T \rangle$  is not minimized among other monomials in $\widehat{f}$ if and only if the degree of $t_{u,1}$ in the respective monomial of the transformed polynomial $\widetilde{f}$ is positive.
	Thus, the existence of toric solutions for the face system $F^{\delta_u}$ is equivalent to existence of toric solutions for the zero evaluation  $\widetilde{F}(\mathbf{t},t_{u,1}=0)$.
	So, if $\widetilde{F}(\mathbf{t},t_{u,1}=0)$ has a solution in $\CC^m$ such that $\widetilde{F}(\mathbf{t},t_{u,1}=t_{{v_1},1}=\cdots=t_{{v_k},1}=0)$ has a solution in $(\CC^*)^m$, then $F^{\delta_u+\delta_{v_1}+\cdots +\delta_{v_k}}$ has a solution in $(\CC^*)^m$.
\end{proof}

The computational gain in this case is that, without the lemma, one would have checked every different combination of the $\delta_u$'s, namely a total of $2^{n-d}$ checks.
Now, it suffices to check only one zero evaluation for each of them, hence only $n-d$ checks.

\paragraph{Third case (mixed normals).}
The third case, that includes the sums of vectors $\delta_u$ and $e_i$ can be also treated with the transformation $\widetilde{F}(\mathbf{t})$ introduced above.
Since the minimization of the inner product is invariant for $d$ of the $d+1$ variables per vertex, the non-existence of zero solutions in $\widetilde{F}(\mathbf{t})$ implies that no $F^w$ has solutions in $(\CC^*)^m$ for all vectors $w$ that are sums of vectors $\delta_u$ with those vectors $e_i$ for which the equality $B \cdot e_i^T=e_i$ holds.
In order to proceed we need the following lemma.
This shows that using $d$ of the $d+1$ $e_i$'s of a vertex suffices to verify if a face system of a mixed normal has solutions in $\CC^m$.

\begin{lemma}\label{lem:alle}
	Let us define a sum of normals 
	$$
	-\delta_u =\sum\limits_{j=1}^{d+1} e_{(d+1)\cdot(u-1)+j} \in \RR^{d+1} .
	$$
	For every $w \in \RR^m$, such that $w$ is a sum of $-\delta_u$ and other normals outside the set\\
	$\{\delta_u,e_{(d+1)\cdot(u-1)+1}, \dots, e_{(d+1)\cdot (u-1)+d+1}\}$ (hence in a subspace complementary to that of $-\delta_u$), the face system $F^{w}$ cannot have a solution in $(\CC^*)^m$.
\end{lemma}
Note that $-\delta_u$ is the sum of $d+1$ normals in complementary subspaces.

\begin{proof}
	We will treat the case of dimension $d=2$ for simplicity of notation; the proof generalizes to arbitrary $d$.
	Without loss of generality, we consider $u=1$.
	Then $w \in \RR^{3}\times \RR^{m-3} $ with $w=(-\delta_1,v)$ and $v \in \RR^{m-3}$.
	The inner products of $-\delta_1$ with the exponent vectors of the magnitude equation in $\RR^2$ for the first coordinate are $2,2,1$, so $f_1^{-\delta_1}=-s$.
	It is obvious that no $w$ which is a sum of $-\delta_1$ and normals not belonging to the set $\{\delta_1,e_{1},  e_{2}, e_3 \}$ defines a face system with no solutions in $(\CC^*)^m$.
	Similarly, the inner products of $-\delta_1$ with the exponent vectors of the magnitude equation $f_1=x_1^2+y_1^2+z_1^2-1$ on $S^2$ are $2,2,2,0$,  yielding $f_1^{-\delta_1}=-1$ which has no solutions in $\CC^m$.	
\end{proof}

Lemma \ref{lem:alle} reveals that in order to verify the conditions of Bernstein's theorem, we can use the transformations $\widetilde{F}(\mathbf{t})$ for all choices of $d$ variables from every set $X_u$, since there is no need to check the cases that include the sum of all $e_i$ normals of a single vertex.
This result, combined with Lemmas~\ref{lem:e_i} \& \ref{lem:Lcase2} leads to the following corollary.

\begin{corollary}\label{cor:allconditions}
	There is a vector $w \in \RR^m$ such that the face system of the sphere equations $F^w$ has a toric root if and only if there is a choice of $\delta$-variables such that the transformed algebraic system $ \widetilde{F}(\bm{t})$ has a zero solution in $\CC^m$ for at least one $\delta$-variable.
\end{corollary}

Since the first $d$ coordinate variables $x_{u,j}$ are symmetric (while $s_u$ variables are not), we can exploit these  symmetries excluding some choices.
So, without loss of generality, we may keep $x_{1,1}$ as a $\delta$-variable from variable set $X_1$, and check all possible choices for $\delta$-variables from all other variable sets $X_u$, such that $u\ne 1$ and $u$ is not among the fixed vertices.

\paragraph{Summary of three cases.} 
In general, if one selects to take into consideration all possible sums of facet normals, then $2^{(d+2)\cdot (n-d)}$ cases should be checked.
We have shown that the category of face systems defined by a sum of coordinate normals cannot have toric solutions, discarding $2^{(d+1)\cdot (n-d)}$ cases.
In the two other cases, the investigation of toric solutions can be combined using the $\widetilde{F}(\mathbf{t})$ transformation.
If a face system has toric solutions, then in the non-coordinate normals case some $\delta$-variables may have zero solutions, while in the mixed normals case both $\delta$-variables and $e$-variables may have zero solutions.
A naive approach to verify these two cases would result to $2^{(d+1)\cdot (n-d)} \cdot (2^{ (n-d)}-1)$ checks, but using Corollary~\ref{cor:allconditions}  one needs to verify the zero evaluations of $\delta$-variables for all choices of $\delta$-variables.
The latter, can be further reduced from $d^{n-d}$ to $d^{n-d-1}$, due to the fact that the coordinate variables are symmetric.
Summarizing, when checking Bernstein's condition, for any of the $d^{n-d-1}$ choices of $\delta$-variable transformation that construct $\widetilde{F}(\mathbf{t})$, it suffices that $n-d$ zero evaluations should be applied for each of the $\delta$-variables.

\begin{theorem}\label{thm:allcases}
	Bernstein's condition can be verified in the case of the sphere equations after checking a total of at most $(n-d)\cdot d^{n-d-1}$ face systems.
\end{theorem}

This discussion yields an effective algorithmic procedure (see Algorithm~\ref{alg:exact}) to verify whether the m-B\'ezout bound is exact.
Function \texttt{ConstructDeltaPoly} takes as input the system of the sphere equations $F$ and a list of $\delta$-variables to construct the polynomials $\widetilde{F}(\mathbf{t})$.
The central role is played by function \texttt{IsmBezoutOfGraphExact}, which verifies if the polynomials $\widetilde{F}$ have zero solutions for the $\delta$-variables.

\begin{algorithm}[htp!]\label{alg:exact}
	\caption{m-B\'ezout exactness for minimally rigid graphs}
	
	\DontPrintSemicolon
	\KwFn{ConstructDeltaPoly}\\
	\KwInput{$F$ (sphere equations),   $V'$ (non-fixed vertices), $L$ (variable indices mapped to $\delta$-variables from each partition set)}
	\KwOutput{$\widetilde{F}$}
	
	\tcc{Transformation to $\widehat{F}(\mathbf{t})$}
	$changevars=\left(\bigcup\limits_{u \in V'} \{x_{u, L(u)} \mapsto \frac{1}{t_{u, L(u)}}\}\right)
	\bigcup 
	\left(\bigcup\limits_{\substack{u \in V' \\ l \in \{1,\cdots, d+1\}\backslash \{L(u)\}}}\{x_{u,l} \mapsto \frac{t_{u,l}}{t_{u,L(u)}}\}\right)$\\
	
	$\widehat{F}(\mathbf{t})=F(changevars)$\\
	$\widetilde{F}(\mathbf{t})=\{\}$\\
	\For{$\widehat{f} \in \widehat{F}(\mathbf{t})$}
	{\tcc{$\alpha(\widehat{f},t_{u,L(u)})=$ (non-positive) degree of $\widehat{f}$ in variable $t_{u,L(u)}$}
		$\widetilde{f}=\prod\limits_{u \in V'} t_{u,L(u)}^{-\alpha(\widehat{f},t_{u,L(u)})} \cdot \widehat{f}$\\
		$\widetilde{F}(\mathbf{t})=\widetilde{F}(\mathbf{t}) \bigcup \{\widetilde{f}\}$
	}
	
	\Return($\widetilde{F}(\mathbf{t})$) \vspace*{3mm}\\
	
	\DontPrintSemicolon
	\KwFn{IsmBezoutOfGraphExact}\\
	\KwInput{$F$ (sphere equations),   $V'$ (non-fixed vertices), Conjecture (If \textbf{True}, only one choice of $\widetilde{F}$, else all choices of $\widetilde{F}$)}
	\KwOutput{\textbf{True} (m-B\'ezout$=c_d(G)$) or \textbf{False} (m-B\'ezout$>c_d(G)$)}
	
	\tcc{Verification of Bernstein's condition}
	\If{Conjecture=\textbf{True}}{
		\tcc{Transformation to $\widetilde{F}(\mathbf{t})$}
		
		$L=[1,1,\dots ,1]$ \tcp*{$L$ length $=$ number of $u \in V'$}
		$\widetilde{F}$=ConstructDeltaPoly(F,V',L)
		
		\tcc{Check for zero solutions of $\delta$-variables (main computation)}
		\For{$u \in V'$}
		{	\tcc{Check if zero evaluation has solutions using ideal of transformed face system}
			\If{$\widetilde{F}(\mathbf{t},\{t_{u,l_u}=0\})$ has a solution}
			{   \Return(\textbf{False})
			}
		}
	}
	\ElseIf{Conjecture=\textbf{False}}
	{
		\For{all choices of 1 out of $d$ variables  in every $X_u$, $u\neq 1$}
		{
			
			$L=[1,(l_u, u \in V'\backslash\{1\})]$ \tcp*{ always same choice for $X_1$}
			$\widetilde{F}$=ConstructDeltaPoly(F,V',L)\\
			
			\tcc{Check for zero solutions of $\delta$-variables (main computation).}
			\For{$u \in V'$}
			{	\If{$\widetilde{F}(\mathbf{t},\{t_{u,l_u}=0\})$ has a solution}
				{   \Return(\textbf{False})
					
				}
			}
		}
	}
	\Return(\textbf{True})
\end{algorithm}

Let us present two examples, further treated in the code found in~\cite{code_mBezout}.

\begin{example}\label{ex:desBer}
	The first (and the simplest non-trivial) example we have treated is an application of our method to the equations that give the embeddings of Desargues graph (double prism) in $\CC^2$ and $S^2$ (see Figure~\ref{fig:Des}).
	It is known that its number of embeddings in $\RR^2$ is 24 \cite{Borcea} and on $S^2$ it is 32 \cite{belt}, and they both coincide with the generic number of complex solutions of the associated algebraic system.
	The m-B\'ezout bound for these systems is $32$, hence it is inexact in the $\CC^2$ case.
	This shall be explained by the fact that the sphere equations in $\CC^2$ have face systems of non-coordinate normals with toric roots.
	
	\begin{figure}[htp]
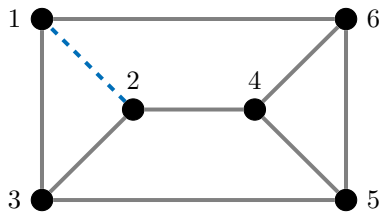
\label{fig:Des}
		\begin{center}
			\Des
			\caption{Desargues graph (double prism)}    
		\end{center}
	\end{figure}
	
	The system of the sphere equations (with vertices~1 and~2 fixed) is a $12 \times 12$ well-constrained system, but we can easily eliminate the linear equations obtaining an 8 by 8 well-constrained system.
	Subsequently, we can also fix vertex~3 up to reflection about the edge $(1,2)$, obtaining finally a system of 6 polynomials in the variables $\{x_4,y_4,x_5,y_5,x_6,y_6 \}$.
	
	If we apply the transformation of variables mentioned above, we can construct a system of polynomials in variables $\{t_{4,1},t_{4,2},t_{5,1},t_{5,2},t_{6,1},t_{6,2} \}$, such that evaluating $t_{4,1},t_{5,1}$ or $t_{6,1}$ to zero corresponds to the face systems of $\delta_4, \delta_5$ or $ \delta_6$ respectively.
	This is one possible choice of $\delta$-variables to construct $\widetilde{F}(\mathbf{t})$.
	Solving these 3 different systems for every $\delta_u$ with Gr\"obner basis in \texttt{Maple} we find the existence of solutions in  $\CC^2$, indicating that the number of complex solutions is strictly smaller than the m-B\'ezout bound. 
	
	In order to get nonzero solutions in $\CC^2$, we need to evaluate to zero all $t_{4,1},t_{5,1},t_{6,1}$ variables, implying that the normal direction for which Bernstein's second theorem shows mixed volume to be inexact is $\left(-1,-1,-1,-1,-1,-1 \right)$.
	This is a normal of a 3-dimensional face, where the face dimension is obtained as 6-3. 
	The normal equals the sum of 3 facet normals.
	
	In the spherical case, no solutions exist, not only for the first choice of $\widetilde{F}(\mathbf{t})$, but also for all the other possible ones (see Algorithm~\ref{alg:exact}), suggesting that the bound is tight, so the number of spherical embeddings is $32$ and equals the m-B\'ezout bound.
	
\end{example}

\begin{example}\label{ex:jacBer}
	The Jackson-Owen graph has the form of a cube with an additional edge adjacent to two antisymmetric vertices (see Figure~\ref{fig:Jac}).
	This graph is the smallest known case that has fewer real than complex embeddings, in $\RR^2$ and $\CC^2$ respectively \cite{Jackson}.
	The m-B\'ezout bound up to the fixed edge shown in the figure is $192$, while the number of embeddings in $\CC^2$ is $90$.
	This shall be explained by a face system of mixed normals that has toric roots. 
	
	\begin{figure}[htp]
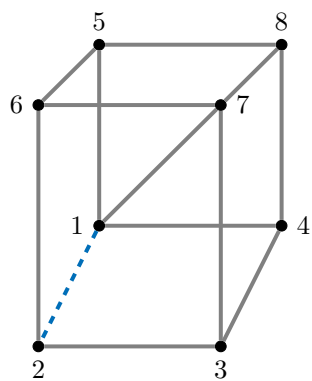

		\begin{center}
			\Jackson
			\caption{The Jackson-Owen graph \label{fig:Jac}}    
		\end{center}
	\end{figure}
	
	The system of the sphere equations is $18 \times 18$, reduced to $13 \times 13$ after linear elimination.
	The set of variables is $\{x_3,y_3,x_4,y_4,x_5,y_5,x_6,y_6,x_7,y_7,x_8,y_8,s_8 \}$.
	
	We apply the transformation of variables, resulting to a new  system of polynomials in variables $\{t_{3,1},t_{3,2},t_{4,1},t_{4,2},t_{5,1},t_{5,2},t_{6,1},t_{6,2},t_{7,1},t_{7,2},t_{8,1},t_{8,2},t_{8,3} \}$.
	As in the case of Desargues' graph, the evaluation of $t_{3,1},t_{4,1},t_{5,1},t_{6,1},t_{7,1},t_{8,1}$ to zero corresponds to the face systems of $\delta_3, \delta_4, \delta_5,\delta_6,\delta_7,\delta_8$ respectively.
	
	We found a solution following zero evaluation of all $\delta$-variables of $\widetilde{F}(\mathbf{t})$  and $t_{8,3}$.
	The normal direction is $\left(-1,-1,-1,-1,-1,-1,-1,-1,-1,-1,-1,-1,0 \right)$ $\in\RR^{13}$.
	This is an example for the mixed normals case, the normal being a sum of one $e_i$ and all 6 $\delta_u$'s.
\end{example}

Another way to apply our method is the computation of a suitable resultant matrix for a given over-constrained system, which follows from evaluating some variable to zero.
It is obvious that if the system has any solutions, the rank of the resultant matrix with sufficiently generic entries is strictly smaller than its size, otherwise it has full rank, assuming we have a determinantal resultant matrix.
We have used \texttt{multires} module for \texttt{Maple} \cite{multires} to examine the existence of solutions, repeating the previous results.
However, these tools seem to be slower than other techniques which directly compute the number of embeddings.\\

In our experimental computations, we noticed that the existence of zero solutions of only one choice of $\widetilde{F}(\mathbf{t})$ (and not all $d^{n-d-1}$) suffice to verify Bernstein's conditions.
Therefore, we state the following conjecture:

\begin{conjecture}\label{con:delta}
	The conditions of Bernstein's second theorem for the system of the sphere equations $F$ are satisfied if and only if the system $\widetilde{F}(\mathbf{t})$ has solutions for every zero evaluation of the $\delta$-variables.
\end{conjecture}

If Conjecture \ref{con:delta} holds, then one only needs to check $n-d$ face systems that correspond to the zero evaluations of $\widetilde{F}(\mathbf{t})$ for each one of the $\delta$-variables, instead of the $(n-d)\cdot d^{n-d-1}$ face systems indicated in Theorem~\ref{thm:allcases}.
Algorithm~\ref{alg:exact} includes the option to consider the Conjecture~\ref{con:delta} to be either True or False.
The first option takes into consideration only one choice of $\delta$-variables, while in the second one all different choices of $\delta$-variables are checked, as in Theorem~\ref{thm:allcases}.

\section{Conclusion and open questions}

We presented new methods to compute efficiently the m-B\'{e}zout bound of the complex embedings of minimally rigid graphs using graph orientations and matrix permanents.
Exploiting existing bounds on planar graph orientations and matrix permanents, we improved the asymptotic upper bounds of the embeddings for planar graphs in dimension 3 and for all graphs in high dimension.
We also compared our experimental results with existing ones indicating that some classes of graphs have tight m-B\'{e}zout bounds.
Finally, we applied Bernstein's second theorem in the case of the m-B\'{e}zout bound for rigid graphs.
Our findings in this topic can be generalized for every class of polynomial systems that have no zero solutions.

Several open questions arise from our results. 
First of all, we would like to improve asymptotic upper bounds,  especially for $d=2$ and $d=3$, exploiting the sparseness of minimally rigid graphs, closing the gap between upper bounds and experimental results.
We would like to examine aspects of both extremal graph theory (for graph orientations) and permanent bounds of matrices with specific structure for this reason.
Besides that, finding the minimal m-B\'ezout bound requires the computation of  bounds up to all possible choices for a fixed $K_d$.
Thus, it would be convenient to find a method to select the $K_d$ that attains the minimum without computing its bound.

Regarding the exactness of the m-B\'ezout bound and the application of Bernstein's second theorem, the first priority would be a possible proof (or refutation) of Conjecture \ref{con:delta}.
Another issue is to optimize this method using the appropriate tools.
A first idea is to construct resultant matrices that exploit the multihomogeneous structure (see for example \cite{EmMantz,DE2003}).
The rank of the matrix could indicate which zero evaluations have solutions for our systems.
Last but not least, we would like to investigate a possible relation between planarity and tight upper bounds, especially in the case of Geiringer graphs.

\paragraph{Acknowledgements}
	IZE and JS are partially supported, and EB is fully supported by project ARCADES which has received funding from the European Union’s Horizon 2020 research and innovation programme under the Marie Sk\l{}odowska-Curie grant agreement No~675789. 
	IZE and EB are members of team AROMATH, joint between INRIA Sophia-Antipolis and NKUA.
	We thank Georg Grasegger for providing us with runtimes of the combinatorial algorithm that calculates $c_d(G)$ and for explaining to EB the counting of embeddings for triangle-free Geiringer graphs. 
	We thank Jan Verschelde for help with applying our method to the Desargues graph, and Timothy Duff for giving us an example on the use of \texttt{MonodromySolver} to find the embeddings of the Icosahedron.

\bibliography{mBezout}

\end{document}